\documentclass[preprint,11pt]{elsarticle}
    \usepackage{amssymb}
     \usepackage{amsthm}
    \usepackage{amsfonts}
    \usepackage{amsmath,amssymb}
    \usepackage[applemac]{inputenc}
    \usepackage{xcolor}
    \usepackage{color, colortbl}
    \usepackage{amsmath}
    \usepackage{amssymb}
    \usepackage{bbm}
    \usepackage{graphicx,floatrow}
    \usepackage{subfigure}
\usepackage{booktabs}
\usepackage[all]{xy}
\usepackage{mathrsfs,bbm}
    \usepackage{tikz}
    \usepackage{geometry}
    \usetikzlibrary{arrows}
    \usepackage{amsfonts}
    \usepackage{mathrsfs}
    \usepackage{amsmath,amssymb}
    \usepackage[applemac]{inputenc}
    \usepackage{color}
    \usepackage{amsmath}
    \usepackage{amssymb}
    \usepackage{graphicx}
    \usepackage{comment}
    \usepackage{tikz}
    \usepackage{hyperref}
    \newtheorem{theorem}{Theorem}[section]
    \newtheorem{lemma}[theorem]{Lemma}
    
    \newtheorem{proposition}[theorem]{Proposition}
    
    \newtheorem{definition}[theorem]{Definition}

    \newtheorem{remark}[theorem]{Remark}

\begin{document}
\newcommand{\A}{\mathbb{A}}
\newcommand{\Q}{\mathbb{Q}}
\newcommand{\N}{\mathbb{N}}
\newcommand{\Z}{\mathbb{Z}}
\newcommand{\R}{\mathbb{R}}
\newcommand{\M}{\mathbb{M}}
\newcommand{\tx}[1]{\quad\mbox{#1}\quad}
\newcommand{\Aut}{\mathrm{Aut}}
\newcommand{\Inn}{\mathrm{Inn}}
\newcommand{\Sym}{\mathrm{Sym}}
\newcommand{\LSec}{\mathrm{LSec}}
\newcommand{\RSec}{\mathrm{RSec}}
\newcommand{\di}{\mathrm{d}}

\newcommand{\giu}{\textcolor{violet}}
\newcommand{\nic}[1]{\textcolor{blue}{#1}}

\begin{frontmatter}
\title{{\bf{Extinction and Persistence in a Stochastic Mpox Model with Hawkes-type Self-Excitation
}
}\tnoteref{label1}}\tnotetext[label1]{}
\author[nor,nhh]{Giulia Di Nunno}
\ead{giulian@math.uio.no}
\author[it]{Nicola Giordano}
\ead{ngiordano@unisa.it}
\author[it]{Barbara Martinucci}
\ead{bmartinucci@unisa.it}
\author[nor,kpi]{Olena Tymoshenko}
\ead{otymoshenkokpi@gmail.com}


\address[nor]{Department of Mathematics, University of Oslo, Norway}
\address[nhh]{Department of Business and Management Science, NHH Norwegian School of Economics, Norway}
\address[it]{Department of Mathematics, University of Salerno}
\address[kpi]{Department of Mathematical Analysis  and Probability Theory, NTUU Igor Sikorsky Kyiv Polytechnic Institute, Ukraine}

\begin{abstract}
We develop a stochastic human-rodent compartment model for Mpox transmission that combines diffusion noise with Hawkes self-exciting jumps in the human infection dynamics. Including Hawkes processes allows, for instance, to model the short but significant spikes in transmission happening after crowded events. For the coupled human-rodent system, we prove global existence, uniqueness and positivity of solutions, derive a basic reproduction number 
$R_0$ that guarantees almost sure extinction when 
$R_0<1$, and obtain explicit persistence-in-the-mean conditions for both infected rodents and humans, which define persistence thresholds for the joint dynamics. Numerical experiments show how clustered human transmission events, environmental variability and control measures shift these thresholds and shape the frequency and size of Mpox outbreaks.

\end{abstract}

\end{frontmatter}
\section{Introduction}

The transmission of infectious diseases is a multi-scale process shaped by biology, behavior and environment. Stochastic differential equations (SDEs) are widely used to capture uncertainties arising from heterogeneity and chance events; unlike deterministic models, SDEs explicitly incorporate random fluctuations and thus provide a more realistic lens on outbreak variability \cite{m1,m2}. Relatedly, ecological population dynamics have long been analyzed with time-series methods \cite{mid1} and stochastic calculus tools \cite{mid2,Agarwal,ghaus3}, yielding insights that translate naturally to host-pathogen systems.
Mpox (monkeypox) is a zoonosis caused by an Orthopoxvirus \cite{Durski,Jezek}. Following the eradication of smallpox, Mpox has become the most prevalent orthopoxvirus infection in humans \cite{Kantele}. Rodents are considered the primary reservoir, with spillover to humans occurring directly or indirectly; primates and other wildlife can contribute to maintenance and spread. Human infection arises through contact with infected animals or people, respiratory droplets, contaminated fluids, or lesion material \cite{Alakunle}. 

The 2022--2024 global Mpox outbreak underscores these challenges. As of 31 October 2024, the World Health Organization reported 115{,}101 confirmed cases (plus 2 probable) and 255 deaths across 126 Member States, with the most recent complete month's reports concentrated in the African Region (71\%) and Western Pacific Region (11.6\%). Cumulatively, the United States, Brazil, the Democratic Republic of the Congo, Spain, France, Colombia, Mexico, the United Kingdom, Germany, and Peru account for approximately 79\% of cases, highlighting marked geographic heterogeneity (WHO, 31 Oct 2024,  \cite{CDC2024}, \cite{who-mpox}).

To date, several studies have addressed the modeling of Mpox transmission using various stochastic and statistical approaches. For instance, one study employed a Bayesian Directed Acyclic Graphic Model to estimate the global spread of the infection. A recent work by Rahman et al.~\cite{Rahman2025} has studied Mpox transmission under stochastic influences, showing how demographic and biological factors interact with randomness.
Their analysis established threshold conditions for disease persistence and extinction, and proved the existence of a unique global nonnegative solution. Another work \cite{MonkeypoxJump} utilized stochastic differential equations of the It\^o-L\'evy type to describe the 2022 outbreak, accounting for cross-species transmission between animals and humans. In \cite{midNew2} analyses of trends and case forecasts have been performed using ARIMA models and join point regression methods. However, despite these advances, the clustered and self-exciting nature of Mpox transmission was not yet included in the modeling. In this work, we directly address this issue and we propose to add a Hawkes  component besides the diffusion in the dynamics of human population. Indeed, the application of self-exciting process models appears to be well-justified and, to date, it is the first time Hawkes processes are used in the modeling transmission of Mpox.

In this study, we use Hawkes processes to model jumps in the human population. Unlike classical Poisson jumps, Hawkes processes are self-exciting: when one event occurs, it temporarily increases the probability of new events happening soon after. This property is important for modeling Mpox because the disease can show clustered outbreaks, where one infection can lead to a sudden increase in new cases. This often happens during mass gatherings, such as concerts, festivals, football matches, or crowded indoor events, where the number of close contacts rises sharply. Such situations can cause short but significant spikes in transmission.
Therefore, the self-exciting mechanism of Hawkes processes provides a more realistic description of how Mpox spreads during these events. In contrast, models with Poisson jumps cannot account for the influence of past events on future transmission. 
In addition, we report that Hawkes processes have been applied in the modeling of diseases with similar clustered patterns of spread. For example, the study by Chiang et al.~\cite{Chiang2022} used a Hawkes model to describe COVID-19 transmission with mobility data while Garetto et al.~\cite{Garetto2021} employed a time-modulated Hawkes process to evaluate the impact of control measures. These results support the idea that self-exciting dynamics are suitable for modeling outbreak clustering in Mpox as well. 
We stress that our work differs from these two not only for the disease considered, but mostly for the scopes of the probabilistic analysis and the modeling choices of using the Hawkes process as a driving noise in the human population.

The inclusion of jumps in compartmental models allows to represent sudden, discontinuous and potentially rare events that significantly influence the epidemic dynamics. While Brownian  noise captures continuous and small fluctuations, jumps model non-continuous events in time.
%
%
%
We did not introduce jumps in the dynamics of the rodent population because, unlike the human population, there is no form of control or direct observability of social behavior for these.
Then the uncertainty associated with transmission between rodents is modeled exclusively with Brownian noise, sufficient to represent natural variability without having to introduce discontinuous jumps.

The human population is divided into four epidemiological compartments: susceptible individuals $S_h$, infected individuals $I_h$, isolated individuals $Q_h$, and recovered individuals $R_h$. The rodent population consists of susceptible rodents $S_r$ and infected rodents $I_r$.

\begin{figure}[!ht]\label{fig}
\center{\includegraphics[scale=0.82]{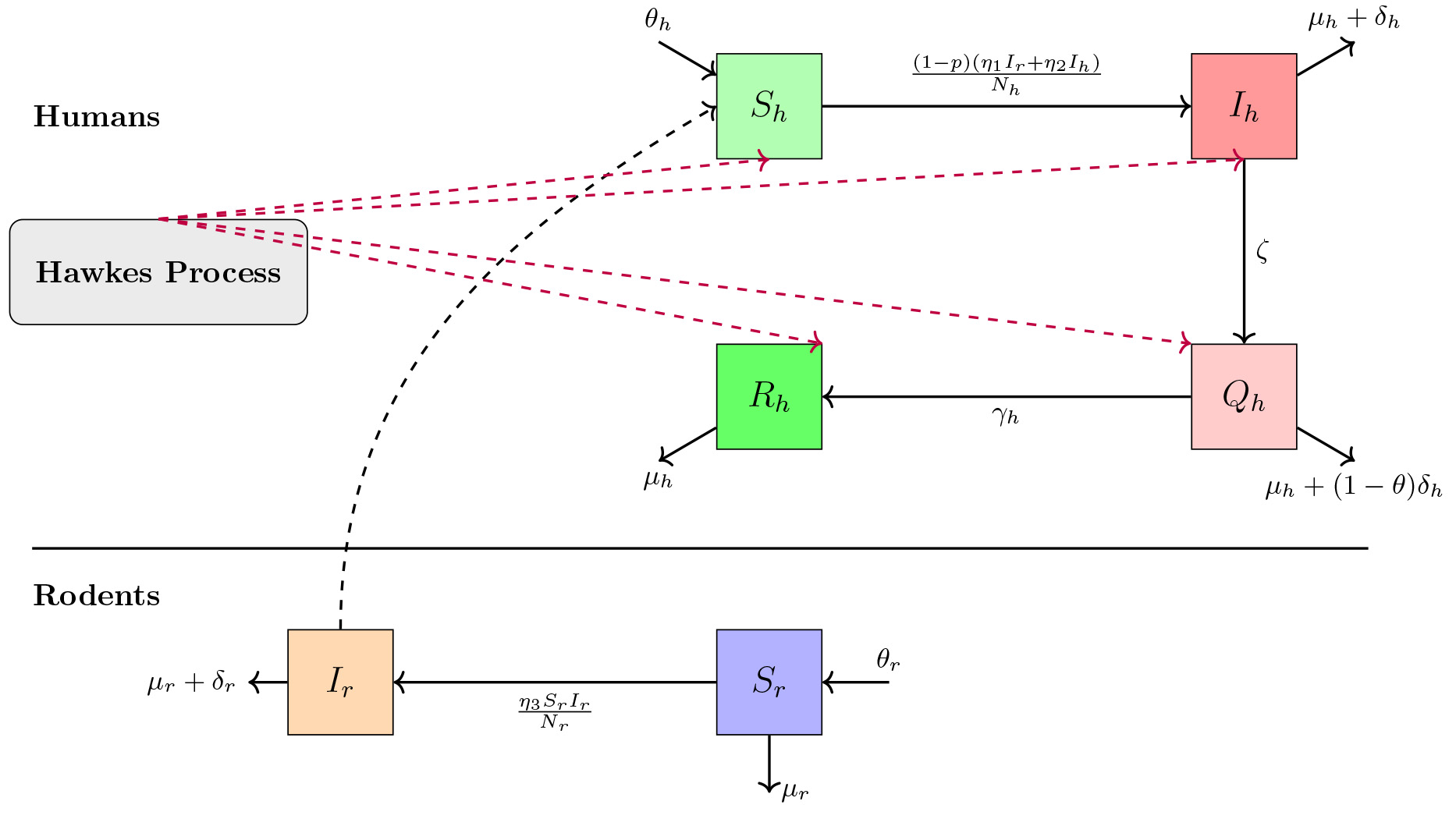}}
\caption{Schematic diagram of Mpox virus transmission}
\end{figure}

To incorporate random environmental fluctuations and uncertainties, such as unpredictable changes in birth/death rates and inaccuracies in counting individuals, we extend the deterministic model by introducing stochastic perturbations via Brownian motions and jump processes.
Specifically, we consider the following stochastic differential system for Mpox dynamics, with human compartments $(S_h, I_h, Q_h, R_h)$ and rodent compartments $(S_r, I_r)$:

\begin{align}\label{System}
d S_{h}(t)=&\left(\theta_{h}-(1-p)\cdot\frac{\eta_1I_r(t-)+\eta_2I_h(t-)}{N_h(t-)}S_{h}(t-)-\mu_{h}S_{h}(t-)\right)dt\notag\\
&-(1-p)\dfrac{\sigma_{1}I_{r}(t-)d B_{1}(t)+\sigma_{2}I_{h}(t-)dB_{2}(t)}{N_{h}(t-)}S_{h}(t-)+\sigma_{3}S_{h}(t-)dB_{3}(t)\notag\\
&+\int_{\mathbb{R}}\epsilon_{1}(y)S_{h}(t-)H_{1}(dt,dy)\notag\\
d I_{h}(t)=&\left((1-p)\cdot\frac{\eta_1I_r(t-)+\eta_2I_h(t-)}{N_h(t-)}S_{h}(t-)-(\mu_{h}+\delta_{h}+\zeta)I_{h}(t-)\right)dt\notag\\
&+(1-p)\dfrac{\sigma_{2}S_{h}(t-)}{N_{h}(t-)}I_{h}(t-)d B_{2}(t)+\sigma_{4}I_{h}(t-)d B_{4}(t)\notag\\
&+\int_{\mathbb{R}}\epsilon_{2}(y)I_{h}(t-)H_{2}(dt,dy)\notag\\
d Q_{h}(t)=&\left(\zeta I_{h}(t-)-(\mu_{h}+\gamma_{h}+(1-\theta)\delta_{h})Q_{h}(t-)\right)dt\notag\\
&+\sigma_{5}Q_{h}(t-)d B_{5}(t)+\int_{\mathbb{R}}\epsilon_{3}(y)Q_{h}(t-)H_{3}(d t,d y)\notag\\
d R_{h}(t)=&\left(\gamma_{h}Q_{h}(t-)-\mu_{h}R_{h}(t-)\right)dt\notag\\
&+\sigma_{6}R_{h}(t-)d B_{6}(t)+\int_{\mathbb{R}}\epsilon_{4}(y)R_{h}(t-)H_{4}(dt,dy)\notag\\
d S_{r}(t)=&\left(\theta_{r}-\frac{\eta_{3}S_{r}(t)I_{r}(t)}{N_{r}(t)}-\mu_{r}S_{r}(t)\right)dt + \sigma_{7}S_{r}(t)d B_{7}(t)\notag\\
d I_{r}(t)=&\left(\frac{\eta_{3}S_{r}(t)I_{r}(t)}{N_{r}(t)}-(\mu_{r}+\delta_{r})I_{r}(t)\right)dt + \sigma_{8}I_{r}(t)d B_{8}(t)
\end{align}
We assume that $(S_h(0), I_h(0), Q_h(0), R_h(0), S_r(0), I_r(0)) \in \mathbb{R}_+^6$; $B_i(t)$ for $i=1,\dots,8$ are mutually independent standard Brownian motions; $\sigma_i$ are volatility coefficients representing the intensity of environmental noise;
$H_i(dt, dy)$ are mutually independent Hawkes processes modeling sudden jumps in the system and $\epsilon_i$ are the jumps' sizes. Moreover, Brownian motions and Hawkes processes are independent of each other.
We define the total human and rodent populations as
\[
N_h(t) := S_h(t) + I_h(t) + Q_h(t) + R_h(t), \quad N_r(t) := S_r(t) + I_r(t),
\]
respectively. Additionally, we assume the following conditions 
\begin{eqnarray}
&& \hspace*{0.1cm}
N_h(t)\leq M,
\label{BTP}
\\
&& \hspace*{0.1cm}
N_r(t) \leq k(t)N_h(t),
\label{HR}
\end{eqnarray}
where $k(t)$ is a positive, continuous function.
Biologically, \eqref{BTP} reflects that the human population evolves in a bounded habitat with limited resources, so that its size cannot grow beyond a regional carrying capacity $M$. 
The inequality \eqref{HR} expresses that the rodent population is uniformly controlled by the human population: synanthropic rodents depend on human food and shelter, and their abundance cannot become arbitrarily larger than $N_h(t)$. 
The factor $k(t)$ is allowed to vary in time to account for seasonal or environmental fluctuations in rodent abundance relative to humans (for example, changes in climate or food availability). 
Analytically, the bounds \eqref{BTP}-\eqref{HR} are used below to control the drift and jump coefficients and to obtain uniform moment estimates for the stochastic system.

\begin{figure}[t]
 \centering
\includegraphics[width=1.1\textwidth]{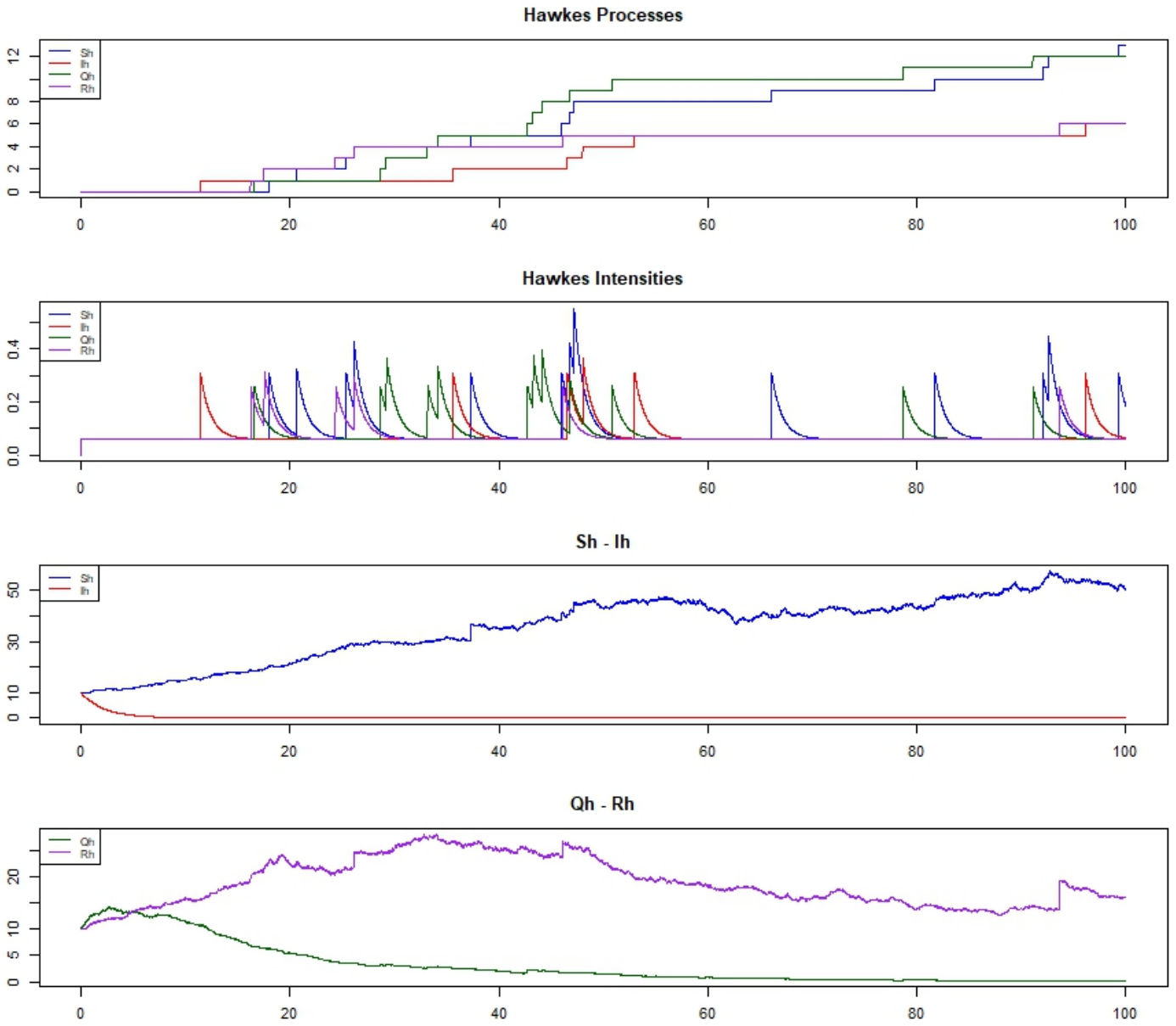}
\caption{The first two graphs show some sample paths of linear Hawkes processes with exponential kernels and their corresponding conditional intensity function. The last two graphs display the plots of the associated variables.}
\label{fig:simulation}
\end{figure}

The stochastic jumps are governed by a compensated Hawkes processes $\tilde{H}_i$, which capture self-exciting dynamics,  that is, the likelihood of future jumps increases after an initial jump. This structure is biologically relevant for modeling clustering effects such as super-spreader events or localized outbreaks caused by one infectious contact. Unlike memoryless Poisson jumps, Hawkes processes introduce temporal dependence and better reflect real-world epidemic surges. 

Note in particular that the rate of contacts, $\eta_1$ and  $\eta_2$, are the basic parameters that are responsible for Mpox disease's transmission. 
\noindent The susceptible human class  $S_h$ is generated by a constant recruitment rate $\theta_h$ and reduced by acquiring infection after interaction with infectious humans and rodents. The force of infection is $\eta_1 I_r+\eta_2 I_h$, where $\eta_1$ and $\eta_2$ are the effective contact rates (capable of transmitting infection) corresponding to infectious rodents and humans, respectively. The mortality rate of each class within the human population decreases over time due to natural deaths occurring at a rate of $\mu_h$. Meanwhile, the death rate caused by Mpox in the infected and isolated human compartments is indicated by $\delta_h$. The infectious human population is isolated and joins the $Q_h$ class at the transmission rate $\zeta$. The recovery rate of the isolated/quarantine human population is denoted by $\gamma_h$.  The value $p$, with $0\leq p\leq 1$, measures the effectiveness of enlightenment campaign,  and $\theta$, with $0 \leq \theta \leq 1$, is the effectiveness of quarantine and
treatment. It is assumed that the death in $Q_h$ due to disease is influenced by the effectiveness of treatment, hence it is $(1 - \theta) \delta_h$ .

In the rodent population, $\theta_r$ is the recruitment rate and $\mu_r$ the natural death rate. Virus transmission between rodents occurs at rate $\eta_3 I_r$, where $\eta_3$ is the transmission coefficient. All rodent classes experience natural mortality at rate $\mu_r$, and infected rodents have an additional Mpox-induced death rate $\delta_r$. Since wild rodents typically have no access to treatment, we assume they do not recover. Model components are summarized in Table \ref{Table 1} (excluding volatilities), and parameter values used in the numerical experiments are taken from Tables \ref{model_params_main}$-$\ref{tab:hawkes_params}.

\begin{table}[H]
  \centering
  \scriptsize
  \renewcommand{\arraystretch}{1.05}
  \begin{tabular}{ll}
    \toprule
    \textbf{Symbol} & \textbf{Description} \\
    \midrule
    $N_h$ & Total human population \\
    $S_h$ & Susceptible humans \\
    $I_h$ & Infected humans \\
    $Q_h$ & Quarantined humans \\
    $R_h$ & Recovered humans \\
    $N_r$ & Total rodent population \\
    $S_r$ & Susceptible rodents \\
    $I_r$ & Infected rodents \\
    $\theta_h$ & Human recruitment rate \\
    $p$ & Public enlightenment campaign effectiveness \\
    $\eta_1$ & Rodent-human contact rate \\
    $\eta_2$ & Human-human contact rate \\
    $\eta_3$ & Rodent-rodent contact rate \\
    $\mu_h$ & Natural human death rate \\
    $\delta_h$ & Disease-induced human death rate \\
    $\zeta$ & Infected-to-quarantine rate \\
    $\gamma_h$ & Recovery rate of quarantined humans \\
    $\theta$ & Quarantine/treatment effectiveness \\
    $\theta_r$ & Rodent recruitment rate \\
    $\mu_r$ & Natural rodent death rate \\
    $\delta_r$ & Disease-induced rodent death rate \\
    \bottomrule
  \end{tabular}
  \caption{Descriptions of the model elements}
 \label{Table 1}
\end{table}

Figure \ref{fig:simulation} provides (1) the cumulative number of arrivals of linear Hawkes processes with exponential kernels and branching parameter less than one; (2) the corresponding conditional intensity functions; (3) the sample paths of the variables $S_h(t)$ and $I_h(t)$; 
(4) the sample paths of the variables $Q_h(t)$ and $R_h(t)$. 


The paper is organized as follows. Section~2 introduces the Hawkes framework and the stochastic calculus with random measures used throughout the paper. Section~3 presents the full human-rodent model with diffusion noise and Hawkes jumps, together with results on global existence, uniqueness, and positivity of solutions to the stochastic dynamical system (\ref{System}). In Section~4, we derive the basic reproduction number and establish an extinction threshold.  This section contains numerical simulations illustrating the theoretical findings. Section~5 provides persistence-in-the-mean results for both human and rodent infections. 

\section{Hawkes Processes and Random Measures}

The stochastic model described in the previous section accounts for various sources of randomness, such as environmental fluctuations (via Brownian motion) and sudden, irregular events (via jump processes). 
We stress that the standard Poisson jump processes assume memoryless behavior, which may not sufficiently capture real-world dynamics, especially in epidemic modeling where jump events, such as infection spikes or super-spreader episodes, often exhibit temporal clustering and feedback mechanisms.

To address this limitation from the epidemiological modeling point of view, we incorporate \textit{Hawkes processes} as the driving mechanism behind the stochastic jumps. Hawkes processes are self-exciting point processes in which the occurrence of one jump increases the probability of subsequent jumps within a short time window. This is particularly relevant for modeling infectious diseases like Mpox, where a single infectious event (e.g., exposure during crowded gathering or rodent interaction in an urban cluster) may trigger further localized outbreaks.
Unlike standard Poisson processes, which treat each event as independent, Hawkes processes allow for self-excitation, where past events increase the likelihood of future events. 
%
%
%
The random jump terms in system \eqref{System} are therefore modeled using \emph{compensated Hawkes random measures}, which allow for the inclusion of memory effects and feedback into the stochastic dynamics. In this section, we present the mathematical background of Hawkes processes and the extension to Hawkes-driven random measures used in the model.
{\color{black}
\subsubsection{Point processes and Hawkes processes}
A simple point process is a stochastic process defined by random event times ${\bf T}={T_1,T_2,\ldots}\subset[0,\infty)$ such that
$P(0<T_1\le T_2\le \cdots)=1.$
Here, $T_i$ is the occurrence time of the $i$-th event.  The counting process associated with $\bf{T}$ is defined as 
$
H_t := \sum_{i \geq 1} \mathbb{I}_{\{t \geq T_i\}},
$ so that $H_t$ counts the number of events in the time interval $[0,t]$, with $H_0=0$. The counting process $H_t$ is a right-continuous step function 
and, due to the assumptions on $T_i$, it increases only in jumps of size $+1$. Note that the counting and point process terminology is usually interchangeable.
The random variable $T_{\infty}:=\sup_i T_i$ is called the explosion time of the point process. If $T_{\infty}=\infty$ almost surely, then 
$N_t$ is said to be nonexplosive. 
\par
Hereafter we provide the definition of univariate  Hawkes processes, introduced by Hawkes in 1971 (see \cite{Hoks}).

\begin{definition}\label{def:hawkes}
Let us consider a counting process $H_t$ and denote by $\mathcal{F}_t=\sigma\{H_s\,|\,0\leq s\leq t\}$ the canonical filtration of $H_t$.
The linear Hawkes process is a self-exciting process for which
$$
\mathbb{P}(H_{t+h}-H_t = n \,|\, \mathcal{F}_t)=\left\{
 \begin{array}{ll}
 1-  \lambda_t h + o(h) & n=0,\\
 \lambda_t h + o(h) & n=1,\\
o(h) & n>1,
\end{array}
 \right.
$$
where
\begin{equation}
	\lambda_t:=\lambda(t | \mathcal{F}_t) = \lim_{h \to 0} \frac{P[H_{t+h}-H_t=1\,|\,\mathcal {F}_t]}{h}
		\label{eq:intensity_def}
\end{equation}
is the (conditional) intensity function of $H_t$.
\par
\noindent
Given the observed event times $\{T_i\}_{i\in {\mathbb N}}$ on the 
interval $[0,t]$, the conditional intensity is 
\begin{equation}
\lambda_t =\lambda_0+ \alpha \int_{0}^t  \nu(t-s) {\rm d}H_s
=\lambda_0+\alpha \sum_{T_i<t} \nu(t-T_i),
\label{eq:hawkes}
\end{equation}
where $\lambda_0>0$ is the baseline intensity, $\alpha>0$ is the mark, and
$\nu: \R_{+}\rightarrow \R_{+}$ is the kernel. 
\end{definition}

\noindent
In general, such process is non-Markovian since its future evolution depends on the timing of past events. In addition, the Hawkes process is nonexplosive if $\int_{0}^t \nu(s) {\rm d}s<\infty$ (see \cite{Bacry2}), 
and it has asymptotically stationary increments (and $\lambda_t$ is asymptotically stationary) if the kernel satisfies the stationary condition $|\nu|_{{\mathcal L}_1}<1$ (see \cite{Bacry1}).
\par
Given a Hawkes process $H_t$,  the integrated conditional intensity function $\Lambda_t=\int_0^t \lambda_s {\rm d}s$ is  
the unique predictable, non-decreasing process such that $Y_t:=H_t-\Lambda_t$ is a local martingale with respect to the canonical filtration of the process $H_t$. 
The process $\Lambda_t$ is called the $\mathcal {F}_t$-compensator of $H_t$. 
%
\par
Various types of kernel functions $\nu(\cdot)$ can be employed in the formulation of Hawkes processes, including exponential, power-law, Gaussian, and piecewise-defined kernels. 
The choice of the kernel determines the memory and excitation structure of the process. A common choice for the excitation function is a decaying exponentially kernel $\nu(t) = \alpha e^{-\beta t}, \alpha, \beta > 0$, 
where $\alpha$ controls the intensity of excitation and $\beta$ determines the rate of decay. In this case, the conditional intensity, say $\lambda^{*}_t$, is given by
\begin{equation}
\lambda^{*}_t =\lambda_0+ \alpha \int_{0}^t  {\rm e}^{-\beta (t-s)} {\rm d}H_s
=\lambda_0+\alpha \sum_{T_i<t} {\rm e}^{-\beta (t-T_i)},
\label{eq:hawkes_exp}
\end{equation}
and describes a scenario in which the probability of experiencing consecutive events is maximal in the period directly after an initial event and gradually decreases as the time elapses. See the simulations in Figure \ref{fig:simulation}.
In this case, the conditional intensity process $\lambda^{*}_t$ and the vector $(H_t,\lambda^{*}_t)$ are both Markov processes (see, for instance \cite{Oakes}). 
Moreover, $\lambda^{*}_t$ satisfies the following stochastic differential equation
$$
{\rm d} \lambda^{*}_t =\beta (\lambda_0- \lambda^{*}_t) {\rm d} t+\alpha {\rm d}H_t,
$$ 
so that $\lambda^{*}_t$ is a mean-reverting process driven by its own point process. 

The exponential Hawkes process is stationary if and only if $\alpha/\beta < 1$ and this condition also prevents the jump intensity from explosion.  
Under such assumption, by taking the expectation of (\ref{eq:hawkes_exp}) and by using \( \mathbb{E}[dH_s] = \mathbb{E}[\lambda^{*}_s] \, {\rm d}s \), we obtain 
$$
\mathbb{E}[\lambda^{*}_t]=\frac{\lambda_0}{1-\alpha/\beta}+\frac{\lambda_0}{1-\beta/\alpha} {\rm e}^{-(\beta - \alpha)t},
$$
and
$$
\mathbb{E}[H_t] = \frac{\lambda_0 t}{1-\alpha/\beta}+\frac{\alpha \lambda_0}{(\alpha - \beta)^2} \left( {\rm e}^{-(\beta - \alpha)t}-1 \right).
$$

\subsection{Stochastic integrals with respect to Hawkes random measures and It\^o formula}

Let us consider a Hawkes random measure $H({\rm d}t,{\rm d}y)$ defined over $[0,T] \times \mathbb{R}^p$, with $p \geq 1$, 
and denote by $\lambda_t$ the corresponding intensity process.
This represents the jump behavior of the Hawkes process $H_t$.  Indeed, for any $s \leq t$ and $A$ measurable in $\mathbb{R}^p$, $H\big((s,t]\times A\big)$ is the number of jumps in the time interval $(s,t]$ of size in $A$ that are taken by the Hawkes process.
The compensated Hawkes random measure $\widetilde{H}({\rm d}t,{\rm d}y)$ is a purely discontinuous martingale difference random measure defined as 
\begin{equation}\label{Eq:Hawkes_rm}
   \widetilde{H}({\rm d}t,{\rm d}y) := H({\rm d}t,{\rm d}y) - \lambda_t\,{\rm d}t\,m({\rm d}y), 
\end{equation}
where $m$ is a probability measure on \( (\mathbb{R}^p, \mathcal{B}(\mathbb{R}^p))\) such that \( m(\{0\})=0 \), describing the jump size. Also, if the expectation exists,
$$E\big[ H\big((s,t]\times A \big] = \int_0^T \int_A d\Lambda_t\, m(dy) = \int_0^T \int_A \lambda_t \, dt \,  m(dy),
$$
with $\Lambda_t$ the predictable $\mathcal{F}_t$-compensator of $H_t$.
\par
For the definition of the stochastic integral with respect to Hawkes random measures, we consider 
the space \( L^2_{\mathcal{F},\lambda m}([0,T]\times\mathbb{R}^p) \), which denotes the set of \( \mathcal{F}_t \)-adapted random fields $\varphi$ such that
$$
\|\varphi\|^2_{L^2_{\mathcal{F},\lambda m}} := \mathbb{E}\left[\int_0^T \int_{\mathbb{R}^p} |\varphi(s,y)|^2 \lambda_s\,m({\rm d}y)\,{\rm d}s\right] < \infty.
$$
Following the approach described in \cite{bensoussan}, for any $\varphi \in L^2_{\mathcal{F},\lambda m}$, it is 
\begin{equation}
\int_0^T \int_{\mathbb{R}^p} \varphi(s,y)\,{H}({\rm d}s,{\rm d}y)=\int_0^T \int_{\mathbb{R}^p} \varphi(s,y) \lambda_s\,m({\rm d}y)\,{\rm d}s+\int_0^T \int_{\mathbb{R}^p} \varphi(s,y)\,\widetilde{H}({\rm d}s,{\rm d}y),
\label{intHaw}
\end{equation}
where, on the right-hand-side, the first term is an ordinary integral, whereas the second one is defined through the isometry
\[
\mathbb{E}\left[\left(\int_0^T \int_{\mathbb{R}^p} \varphi(s,y)\,\widetilde{H}({\rm d}s,{\rm d}y)\right)^2\right] = \mathbb{E}\left[\int_0^T \int_{\mathbb{R}^p} |\varphi(s,y)|^2 \lambda_s\,m({\rm d}y)\,{\rm d}s\right].
\]

Moreover, for a general stochastic process 
\begin{equation}
\label{eq:hawkes_ito_process}
\xi(t) = \xi(0) + \int_0^t \alpha(s)\,{\rm d}s + \int_0^t \beta(s) \cdot {\rm d}B(s) + \int_0^t \int_{\mathbb{R}^p} \varphi(s,y)\,H({\rm d}s,{\rm d}y),
\end{equation}
where $\varphi \in L^2_{\mathcal{F},\lambda m}((0,T) \times \mathbb{R}^p)$, $H$ is a Hawkes random measure, $B$ is a standard $n$-dimensional Brownian motion, and where $ \alpha \in L^2_{\mathcal{F}}(0,T))$ and 
$\beta \in L^2_{\mathcal{F}}((0,T);\mathbb{R}^n)$ are $\mathcal{F}_t$-adapted processes, the following It\^o Formula is proved in \cite{bensoussan}.

\begin{proposition}[It\^o Formula for Hawkes-It\^o Processes \cite{bensoussan}] 
\label{prop1}
Let $\Psi \in C^{2,1}(\mathbb{R} \times [0,T])$ be such that
$$
  \mathbb{E}\left[\int_0^T \left|\frac{\partial \Psi}{\partial x}(\xi(s),s)\right|^2 |\beta(s)|^2\,{\rm d}s\right] < \infty,
$$
where $|\beta(s)|$ denotes the Euclidean norm in $\mathbb{R}^n$, and
$$
\mathbb{E}\left[\int_0^T \int_{\mathbb{R}^p} \left|\Psi(\xi(s)+\varphi(s,y),s) - \Psi(\xi(s),s)\right|^2 \lambda_s\,m({\rm d}y)\,{\rm d}s\right] < \infty.
$$
Then, for all $t \in [0,T]$,
\begin{eqnarray}
&& \hspace*{-0.6cm}
\Psi(\xi(t),t) = \Psi(\xi(0),0)+ \int_0^t \frac{\partial \Psi}{\partial x}(\xi(s),s) \, \beta_s \cdot {\rm d}B(s)
\nonumber 
\\ 
&& \hspace*{-0.6cm}
+\int_0^t \left( \frac{\partial \Psi}{\partial s}(\xi(s),s) + \alpha(s) \frac{\partial \Psi}{\partial x}(\xi(s),s) + \frac{1}{2} \frac{\partial^2 \Psi}{\partial x^2}(\xi(s),s) |\beta(s)|^2 \right) {\rm d}s 
\nonumber 
\\
&& \hspace*{-0.6cm}
 + \int_0^t \int_{\mathbb{R}^p} \left[ \Psi(\xi(s)+\varphi(s,y),s) - \Psi(\xi(s),s) \right] H({\rm d}s,{\rm d}y).
\label{Ito1}
\end{eqnarray}
\end{proposition}
\noindent
Observe that, thanks to Eq. (\ref{intHaw}), we can write Eq. (\ref{Ito1}) in terms of the compensated Hawkes random measure
\begin{eqnarray}
&& \hspace*{-0.8cm}
\Psi(\xi(t),t) = \Psi(\xi(0),0) + \int_0^t \mathcal{L}_s \Psi(\xi(s),s)\,{\rm d}s+ \int_0^t \frac{\partial \Psi}{\partial x}(\xi(s),s)\, \beta(s)\, {\rm d}B(s) 
\nonumber 
\\
&& \hspace*{0.8cm}
 + \int_0^t \int_{\mathbb{R}^p} \left( \Psi(\xi(s)+\varphi(s, y),s) - \Psi(\xi(s),s) \right) \widetilde{H}({\rm d}s,{\rm d}y),
 \label{Ito2}
\end{eqnarray}
where the operator \( \mathcal{L}_s \) acts as
\begin{eqnarray}
&& \hspace*{-1cm}
\mathcal{L}_s \Psi(x,s) = \frac{\partial \Psi}{\partial s}(x,s) + \alpha_s \frac{\partial \Psi}{\partial x}(x,s) + \frac{1}{2} |\beta_s|^2 \frac{\partial^2 \Psi}{\partial x^2}(x,s)
\nonumber 
\\
&& \hspace*{0.4cm}
+ \int_{\mathbb{R}^p} \left( \Psi(x+\varphi(s,y),s) - \Psi(x,s) \right) \lambda_s m(dy).
\label{operatorL}
\end{eqnarray}

Note that, under the assumption $\int_{0}^t \nu(s) {\rm d}s<\infty$, for any $t>0$ the Hawkes process has almost surely a finite number of jumps on $[0,t]$ 
(see also Theorem $6$ of \cite{Delattre}), so that Eqs. (\ref{Ito2}) and (\ref{operatorL}) agree with the result in Proposition $8.13$ of \cite{Tonk}.

\section{Global Existence and Positivity of the Solution}
\noindent  
In studying dynamic behavior, it is crucial to establish whether a solution exists globally, i.e., there is no explosion in a finite time. In addition, in the context of epidemic population dynamics, it is also important to assess whether the solutions are non-negative. 

A necessary condition for a stochastic differential equation to have a unique  global solution  for any given initial value, is that the coefficients have the linear growth, and are locally Lipschitz (e.g. see \cite{Mao}, \cite{Arnold}, \cite{GS}).
 However, although the coefficients of system (\ref{System}) do not satisfy  linear growth condition due to the bilinear incidence, they are locally Lipschitz continuous. Nonetheless, the following result demonstrates that the solution  of system (\ref{System}) exists globally, it is unique and non-negative.
For the purpose of proving the following theorem, we consider an equivalent formulation of system \eqref{System}, in which the jump terms are expressed via the compensated Hawkes random measure  $ \tilde{H}_i $ instead of the original measures $ H_i $. 
Then the dynamical system reads as follows

\begin{eqnarray}
\label{System1}
&& \hspace*{-1.3cm}
d S_{h}(t)=\bigg(\theta_{h}-(1-p)\cdot\frac{\eta_1I_r(t-)+\eta_2I_h(t-)}{N_h(t-)}S_{h}(t-)-\mu_{h}S_{h}(t-)
\notag\\
&& \hspace*{-0.1cm}
+\int_{\mathbb{R}}S_{h}(t-)\epsilon_{1}(y)\lambda_{t,1}m(dy)\bigg)dt
\notag\\
&& \hspace*{-0.1cm}
-(1-p)\dfrac{\sigma_{1}I_{r}(t-)d B_{1}(t)+\sigma_{2}I_{h}(t-)dB_{2}(t)}{N_{h}(t-)}S_{h}(t-)+\sigma_{3}S_{h}(t-)dB_{3}(t)
\notag\\
&& \hspace*{-0.1cm}
+\int_{\mathbb{R}}\epsilon_{1}(y)S_{h}(t-)\tilde{H}_{1}(dt,dy)
\notag\\
&& \hspace*{-1.3cm}
d I_{h}(t)=\bigg((1-p)\cdot\frac{\eta_1I_r(t-)+\eta_2I_h(t-)}{N_h(t-)}S_{h}(t-)-(\mu_{h}+\delta_{h}+\zeta)I_{h}(t-)
\notag\\
&& \hspace*{-0.1cm}
+\int_{\mathbb{R}}I_{h}(t-)\epsilon_{2}(y)\lambda_{t,2}m(dy)\bigg)dt
\notag\\
&& \hspace*{-0.1cm}
+(1-p)\dfrac{\sigma_{2}S_{h}(t-)}{N_{h}(t-)}I_{h}(t-)d B_{2}(t)+\sigma_{4}I_{h}(t-)d B_{4}(t)
\notag\\
&& \hspace*{-0.1cm}
+\int_{\mathbb{R}}\epsilon_{2}(y)I_{h}(t-)\tilde{H}_{2}(dt,dy)
\notag\\
&& \hspace*{-1.3cm}
d Q_{h}(t)=\left(\zeta I_{h}(t-)-(\mu_{h}+\gamma_{h}+(1-\theta)\delta_{h})Q_{h}(t-)+\int_{\mathbb{R}}Q_{h}(t-)\epsilon_{3}(y)\lambda_{t,3}m(d y)\right)dt
\notag\\
&& \hspace*{-0cm}
+\sigma_{5}Q_{h}(t-)d B_{5}(t)+\int_{\mathbb{R}}\epsilon_{3}(y)Q_{h}(t-)\tilde{H}_{3}(d t,d y)
\notag\\
&& \hspace*{-1.3cm}
d R_{h}(t)=\left(\gamma_{h}Q_{h}(t-)-\mu_{h}R_{h}(t-)+\int_{\mathbb{R}}R_{h}(t-)\epsilon_{4}(y)\lambda_{t,4}m(d y)\right)dt
\notag\\
&& \hspace*{-0.1cm}
+\sigma_{6}R_{h}(t-)d B_{6}(t)+\int_{\mathbb{R}}\epsilon_{4}(y)R_{h}(t-)\tilde{H}_{4}(dt,dy)
\notag\\
&& \hspace*{-1.3cm}
d S_{r}(t)=\left(\theta_{r}-\frac{\eta_{3}S_{r}(t)I_{r}(t)}{N_{r}(t)}-\mu_{r}S_{r}(t)\right)dt + \sigma_{7}S_{r}(t)d B_{7}(t)
\notag\\
&& \hspace*{-1.3cm}
d I_{r}(t)=\left(\frac{\eta_{3}S_{r}(t)I_{r}(t)}{N_{r}(t)}-(\mu_{r}+\delta_{r})I_{r}(t)\right)dt + \sigma_{8}I_{r}(t)d B_{8}(t),
\end{eqnarray}
with $ \tilde{H}_i(dt,dy) = H_i(dt,dy) - \lambda_{t,i} m(dy)dt $, for $ i=1,2,3,4 $, defined as in \eqref{Eq:Hawkes_rm} where $ \lambda_{t,i} $ is the conditional intensity of the jump process $H_i$ as in Definition \ref{def:hawkes}, and 
 $m(\cdot)$ is the probability measure which describes the distribution of the jump sizes, such that $m(0)=0$. Here we have assumed that the jump size distribution is associated to the human population as a whole and not to the compartmental classes.

 In this setting, the system of equations governing the dynamics of the population incorporates self-exciting jump processes, where the jumps are governed by Hawkes processes. 
The presence of these jumps introduces additional complexity into the analysis, and hence, to prove the existence and uniqueness of a solution, we need to establish conditions that ensure the system is well-posed, i.e., 
the solution remains positive for all the times. 
\par
Following Definition \ref{def:hawkes}, here and in the sequel, we fix the (conditional) intensity function of the Hawkes processes $H_i$ $i=1,\ldots,4$, to be 
$
\lambda_i(t)=\lambda_{0,i}+\int_{0}^{t}\nu_i (t-s)
{\rm d} H_i(s),
$
with 
$\nu_i: \R_{+}\rightarrow \R_{+}$ and $\lambda_{0,i}>0$.
\begin{theorem}
\label{Global}
Let $\mathcal{X}(t)=(S_h(t),I_h(t),Q_h(t),R_h(t),S_r(t), I_r(t))$ satisfy system \eqref{System1} with initial condition $\mathcal{X}(0)\in\R^6_{+}$.  Suppose that, for $ i=1,\dots,4$, the following 
assumptions hold
	\begin{itemize}
		\item [(i)] $\int_{\R}\epsilon_{i}(y)m(\di y)<\infty$;
		\item[(ii)] $\epsilon_{i}(y)>0 $;

\item [(iii)] $\int_{0}^{+\infty} \nu_i(s) {\rm d}s<1$.
	\end{itemize}	
Then the system \eqref{System1} admits a unique global solution $\mathcal{X}(t)$ on $[0,\infty)$, and
\[
\mathbb{P}\bigl\{\mathcal{X}(t)\in\R^6_{+}\text{ for all }t\ge0\bigr\}=1.
\]
\end{theorem}

\begin{proof} Since the coefficients of the equation are locally Lipschitz continuous for any given initial value $\mathcal{X}(0)\in \mathbb{R}^6_+$, recalling the theory of stochastic differential equations there is a unique local solution $\mathcal{X}(t)$ on $t \in [0,\tau_e)$, where $\tau_e$ is the explosion time (see \cite{applebaum}).\\
To prove that this solution is global, we have to show  that $\tau_{e}=\infty$ a.s.

This can be done by setting a sufficiently large value for $m_0\geq 0$, such that $\mathcal{X}(0) \in \left[\frac{1}{m_0}, m_0\right]^{\times 6}$. Then we define a stopping time for  each integer $m \geq m_0$:
\[
\tau_m=\inf\Bigl\{t\in[0,\tau_e):\ \exists X\in\{S_h,I_h,Q_h,R_h,S_r,I_r\}\ \text{such that}\ X(t)\notin(1/m,m)\Bigr\}.
\]
 By definition, $\tau_m$ is increasing as  $m\to \infty$, and we set $\tau_\infty=\underset{m \to \infty}{\lim} \tau_m$, then $\tau_\infty \leq \tau_e$  a.s.
To complete the proof, we have to show that $\tau_ \infty=\infty$ a.s. If this statement is false, then we can find a pair of constants $T > 0$ and $\varepsilon \in (0,1)$ such that 
$\mathrm{P}\{\tau_\infty \leq T\} > \varepsilon$. 
Hence, there is an integer $m_1 \geq m_0$ such that
\begin{eqnarray}\label{eq3.1}
\mathrm{P}\{\tau_m \leq T\} \geq \varepsilon, \,\, \textnormal{for all} \,\,\, m \geq m_1.
\end{eqnarray}
Let us set \( A_m = \{\tau_m \leq T\} \). On the event $ A_m $, the definition of the stopping time \( \tau_m \) guarantees that at least one of the six state variables \( S_h(t), I_h(t), Q_h(t), R_h(t), S_r(t), I_r(t) \) reaches either the upper boundary \( m \) or the lower boundary $ 1/m $ at time $ \tau_m $.

\noindent Since \( P(A_m) \geq \varepsilon \), it follows that for all sufficiently large \( m \), there is a positive probability that at least one state variable hits either the upper boundary \( m \) or the lower boundary \( 1/m \). This ensures that \( A_m \) captures the event where the system dynamics push at least one variable outside the bounded region defined by \( (1/m, m) \), as required by the stopping time condition.

\noindent 
Let us now consider the following  function $\mathcal{V}:\mathbb{R}^6_+\to \mathbb{R}_+$:
\begin{eqnarray*}
&&\mathcal{V}(S_h,I_h,Q_h,R_h,S_r,I_r)=(S_h-1-\log
S_h)+(I_h-1-\log I_h)\cr\cr&&+(Q_h-1-\log
Q_h)+(R_h-1-\log R_h)+(S_r-1-\log S_r)+(I_r-1-\log I_r).
\end{eqnarray*}
This function is non-negative, since $x-1-\log x\geq0$ for all $x>0$. By It\^o's formula, 
\begin{align}\label{leeq3.11}
&d\mathcal{V}(\mathcal{X}(t))= L\mathcal{V}(\mathcal{X}(t))dt\notag\\
&
\textcolor{black}{ 
+\int_{\mathbb{R}}[S_{h}(t-)\epsilon_{1}(y)-\log(1+\epsilon_{1}(y))]\lambda_{t,1}m(\di y)+\int_{\mathbb{R}}[I_{h}(t-)\epsilon_{2}(y)-\log(1+\epsilon_{2}(y))]\lambda_{t,2}m(\di y)}\notag\\
&\textcolor{black}{ 
+\int_{\mathbb{R}}[Q_{h}(t-) \epsilon_{3}(y)-\log(1+\epsilon_{3}(y))]\lambda_{t,3}m(\di y)+\int_{\mathbb{R}}[R_{h}(t-)\epsilon_{4}(y)-\log(1+\epsilon_{4}(y))]\lambda_{t,4}m(\di y)}\notag\\
&+\frac{(1-p)\cdot(\sigma_1I_r(t-)S_h(t-)dB_1(t)+\sigma_2I_h(t-)S_h(t-)dB_2(t))}{N_h(t-)}\cdot\left(\frac{1}{S_h(t-)}-\frac{1}{I_h(t-)}\right)\notag\\
&+(S_h(t-)- 1)\sigma_3dB_3(t)
+(I_h(t-)-1)\sigma_4dB_4(t)+(Q_h(t-)-1)\sigma_5dB_5(t)\notag\\
&+(R_h(t-)-1)\sigma_6dB_6(t)+(S_r(t-)-1)\sigma_7dB_7(t)+(I_r(t-)-1)\sigma_8dB_8(t)\notag\\
&+\int_{\mathbb{R}}[S_{h}(t-)\epsilon_{1}(y)-\log(1+\epsilon_{1}(y))]\tilde{H}_{1}(d t,d y)+\int_{\mathbb{R}}[I_{h}(t-)\epsilon_{2}(y)-\log(1+\epsilon_{2}(y))]\tilde{H}_{2}(d t,d y)\notag\\
&+\int_{\mathbb{R}}[Q_{h}(t-)\epsilon_{3}(y)-\log(1+\epsilon_{3}(y))]\tilde{H}_{3}(d t,d y)+\int_{\mathbb{R}}[R_{h}(t-)\epsilon_{4}(y)-\log(1+\epsilon_{4}(y))]\tilde{H}_{4}(d t,d y),
\end{align}
where $L\mathcal{V}:\mathbb{R}^6_+\to \mathbb{R}_+$ is defined by
\begin{eqnarray*}
&&L\mathcal{V}(\mathcal{X}(t))=\theta_h+\theta_r+4\mu_h+\delta_h+\zeta+2\mu_r+\delta_r+\gamma_h+(1-\theta)\delta_h\cr\cr&&+(1-p)\cdot\frac{(\eta_1I_r(t)+\eta_2I_h(t))}{N_h(t)}+\frac{\eta_3I_r(t)}{N_r(t)}-\mu_hS_h(t)-(\mu_h+\delta_h)I_h(t)\cr\cr&&-(\mu_h+(1-\theta)\delta_h)Q_h(t)-\mu_hR_h(t)-\mu_rS_r(t)
-(\mu_r+\delta_r)I_r(t)-\frac{\theta_h}{S_h(t)}\cr\cr&&-(1-p)\cdot\frac{(\eta_1I_r(t)+\eta_2I_h(t))S_h(t)}{N_h(t)I_h(t)}-\frac{ \zeta I_h(t)}{Q_h(t)}-\frac{\gamma_hQ_h(t)}{R_h(t)}-\frac{\theta_r}{S_r(t)}-\frac{\eta_3S_r(t)}{N_r(t)}\cr\cr&&+(1-p)^2\cdot\frac{\sigma^2_1I^2_r(t)+\sigma^2_2I^2_h(t)}{2N^2_h(t)}+(1-p)^2\cdot\frac{\sigma^2_2S^2_h(t)}{2N^2_h(t)}\cr\cr&&+\frac{\sigma^2_3}{2}+\frac{\sigma^2_4}{2}+\frac{\sigma^2_5}{2}+\frac{\sigma^2_6}{2}+\frac{\sigma^2_7}{2}+\frac{\sigma^2_8}{2}.
\end{eqnarray*}
Following Assumption (\ref{HR}), we have that $\dfrac{I_r}{N_h}\leq \bar{k}$ and $S_h< N_h$, $I_r< N_r$, $S_r< N_r$,  $I_h< N_h$, so that we have
\begin{eqnarray}\label{leq3.21}
&L\mathcal{V}(\mathcal{X}(t))&\leq K,
\end{eqnarray}
where the constant $K$ is given by \begin{eqnarray*}
K&=&\theta_h+\theta_r+4\mu_h+\delta_h+\zeta+2\mu_r+\delta_r+\gamma_h+(1-\theta)\delta_h+(1-p)\cdot (\eta_1\bar{k}+\eta_2)\cr\cr&&+\eta_3+(1-p)^2\cdot\frac{\sigma^2_1}{2}\cdot{\bar{k}}^2+(1-p)^2\sigma^2_2+\frac{\sigma^2_3}{2}+\frac{\sigma^2_4}{2}+\frac{\sigma^2_5}{2}+\frac{\sigma^2_6}{2}+\frac{\sigma^2_7}{2}+\frac{\sigma^2_8}{2}.
\end{eqnarray*}

\noindent Integrating both sides of Eq. (\ref{leeq3.11}) from 0 to
$\tau_m \wedge t=\min\{\tau_m,t\}$, where $t\in[0,T]$, and making use 
of Eq. (\ref{leq3.21}), we obtain the following relation
\begin{align*}
\int^{\tau_m \wedge t}_0d\mathcal{V}\bigg(\mathcal{X}(s)\bigg) \leq&
\int^{\tau_m \wedge
t}_0 Kds +
\textcolor{black}{\int_0^{\tau_m \wedge t}\int_{\mathbb{R}}[S_h(s-) \epsilon_{1}(y)-\log(1+\epsilon_{1}(y))]\lambda_{s,1}m(\di y)ds
} \\
& + \textcolor{black}{\int_0^{\tau_m \wedge t}\int_{\mathbb{R}}[I_h(s-) \epsilon_{2}(y)-\log(1+\epsilon_{2}(y))]\lambda_{s,2}m(\di y)ds
} \\
& + \textcolor{black}{\int_0^{\tau_m \wedge t}\int_{\mathbb{R}}[Q_h(s-) \epsilon_{3}(y)-\log(1+\epsilon_{3}(y))]\lambda_{s,3}m(\di y)ds
} \\
& + \textcolor{black}{\int_0^{\tau_m \wedge t}\int_{\mathbb{R}}[R_h(s-) \epsilon_{4}(y)-\log(1+\epsilon_{4}(y))]\lambda_{s,4}m(\di y)ds
} \\
&-(1-p)\int^{\tau_m \wedge
t}_0\frac{(S_h(s-)-1)\sigma_1I_r(s-)}{N_h(s-)}dB_1(s)\\
&+(1-p)\int^{\tau_m \wedge
t}_0\frac{(I_h(s-)-S_h(s-))\sigma_2}{N_h(s-)}dB_2(s)\\
&+
\int^{\tau_m \wedge
t}_0
(S_h(s-)-1)\sigma_3dB_3(s-)+\int^{\tau_m \wedge
t}_0(I_h(s-)-1)\sigma_4dB_4(s)\\
&+\int^{\tau_m \wedge
t}_0
(Q_h(s-)-1)\sigma_5dB_5(s)+\int^{\tau_m \wedge
t}_0
(R_h(s-)-1)\sigma_6dB_6(s)\\
&+\int^{\tau_m \wedge
t}_0 (S_r(s-)-1)\sigma_7dB_7(s)+\int^{\tau_m \wedge
t}_0 (I_r(s-)-1)\sigma_8dB_8(s)\\
&+\int_{0}^{\tau_m \wedge
t}\int_{\mathbb{R}}[S_{h}(s-)\epsilon_{1}(y)-\log(1+\epsilon_{1}(y))]\tilde{H}_{1}(d s,d y)\\
&+\int_{0}^{\tau_m \wedge
t}\int_{\mathbb{R}}[I_{h}(s-)\epsilon_{2}(y)-\log(1+\epsilon_{2}(y))]\tilde{H}_{2}(d s,d y)\\
&+\int_{0}^{\tau_m \wedge
t}\int_{\mathbb{R}}[Q_{h}(s-)\epsilon_{3}(y)-\log(1+\epsilon_{3}(y))]\tilde{H}_{3}(d s,d y)\\
&+\int_{0}^{\tau_m \wedge
t}\int_{\mathbb{R}}[R_{h}(s-)\epsilon_{4}(y)-\log(1+\epsilon_{4}(y))]\tilde{H}_{4}(d s,d y),
\end{align*}
since $\mathcal{X}(\tau_m \wedge
t)\in \mathbb{R}^6_+$, $t\in[0,T]$.
The remainder of the proof follows the arguments presented in Mao et al. \cite{Mao2}. 

Taking the expectation on both sides of the above equation, we have
\begin{eqnarray}
&& \hspace*{-1cm}
\mathrm{E}\mathcal{V}\left(\mathcal{X}(\tau_m\wedge T)\right)\leq \mathcal{V}(\mathcal{X}(0))+ \mathrm{E}\int^{\tau_m \wedge
T}_0 Kds
\nonumber\\
&& \hspace*{-1cm}
+{\mathrm{E} \int_0^{\tau_m \wedge T}\int_{\mathbb{R}}[S_h(s-) \epsilon_{1}(y)-\log(1+\epsilon_{1}(y))]\lambda_{s,1}m(\di y)ds
} 
\nonumber \\
&& \hspace*{-1cm}
 + {\mathrm{E} \int_0^{\tau_m \wedge T}\int_{\mathbb{R}}[I_h(s-) \epsilon_{2}(y)-\log(1+\epsilon_{2}(y))]\lambda_{s,2}m(\di y)ds
} 
\nonumber \\
&& \hspace*{-1cm}
 + {\mathrm{E} \int_0^{\tau_m \wedge T}\int_{\mathbb{R}}[Q_h(s-) \epsilon_{3}(y)-\log(1+\epsilon_{3}(y))]\lambda_{s,3}m(\di y)ds
} 
\nonumber \\
&& \hspace*{-1cm}
 +{\mathrm{E} \int_0^{\tau_m \wedge T}\int_{\mathbb{R}}[R_h(s-) \epsilon_{4}(y)-\log(1+\epsilon_{4}(y))]\lambda_{s,4}m(\di y)ds.
 }
\label{rel11}
\end{eqnarray}

Due to Eq. (\ref{BTP}), we have that
\begin{eqnarray*}
&& \hspace*{-1.8cm}
	\left| \mathrm{E} \int_0^{\tau_m \wedge T}\int_{\mathbb{R}}[S_h(s-) \epsilon_{1}(y)-\log(1+\epsilon_{1}(y))]\lambda_{s,1}m(\di y) {\rm d}s \right| 
	\\
&& \hspace*{-0.7cm}	
\leq (M+1) \left| \mathrm{E} \int_0^{T}  \lambda_{s,1}{\rm d}s \int_{\mathbb{R}} \epsilon_{1}(y) m(\di y) \right|.
	\end{eqnarray*}
	Being $ \int_{0}^{+\infty}\nu(s) {\rm d}s <1 $, there exists a constant $ K_{1}>0 $ such that $ \mathrm{E}[\lambda_{s,1}]<K_{1}$, so that it finally results
\begin{equation*}
	\left| \mathrm{E} \int_0^{\tau_m \wedge T}\int_{\mathbb{R}}[S_h(s-) \epsilon_{1}(y)-\log(1+\epsilon_{1}(y))]\lambda_{s,1}m(\di y) {\rm d}s \right| \leq K_1 \cdot T \cdot C_{1},
\end{equation*}	
where we have set 
$C_{1}:=(M+1)\Big( \int_{\mathbb{R}} \epsilon_{1}(y) m(\di y) \Big). $

The same reasoning can be applied to the other integrals involving the variables $I_h$, $Q_h$ and $R_h$.
Hence, substituting into Eq. (\ref{rel11}), we have that
$
\mathrm{E}\mathcal{V}\left(\mathcal{X}(\tau_m\wedge T)\right) \leq \mathcal{V}(\mathcal{X}(0)) + \hat{K} T,
$
where $\hat{K}$ is a positive constant.

\noindent Recalling that $A_m=\{\tau_m \leq t\}$, for $m \geq m_1$, 
due to Eq. (\ref{eq3.1}), we have $\mathrm{P}(A_m) \geq \varepsilon$.
Note that for every $\omega \in A_m$, there is at least one term among $S_h(\tau_m\wedge
T)$, $I_h(\tau_m\wedge T)$, $Q_h(\tau_m\wedge
T)$, $R_h(\tau_m\wedge T)$, $S_r(\tau_m\wedge T)$ and $I_r(\tau_m\wedge T)$ which is
greater or equal to $m$ or lower or equal to $\dfrac{1}{m}$. Then
$$\mathcal{V}(\mathcal{X}(\tau_m\wedge T)) \geq
\Bigg(\left(m-1-\log m\right)\wedge\left(\frac{1}{m}-1-\log \frac{1}{m}\right)\Bigg).
$$
Therefore, from Eqs. (\ref{eq3.1}) and (\ref{rel11}),  we can easily obtain
\begin{eqnarray*}
\mathcal{V}(\mathcal{X}(0))+KT \geq \mathrm{E}\Bigg(\mathbb{I}_{A_m(\omega)}
\mathcal{V}(\mathcal{X}(\tau_m))\Bigg)
\geq\varepsilon\Bigg(\left(m-1-\log m\right)\wedge\Bigg(\frac{1}{m}-1-\log \frac{1}{m}\Bigg)\Bigg),
\end{eqnarray*}
where $\mathbb{I}_{A_m(\omega)}$  is the indicator function of $A_m$.  Letting $m$ increase to infinity, we find a   contradiction:
$\infty>\mathcal{V}(\mathcal{X}(0))+\hat{K}T=\infty.$ Thus, the assumption
$\mathrm{P} \{\tau_{\infty}<\infty\}>0$
must be false, and  $\tau_{\infty}=\infty$ a.s.
The non-negativity of the solution is guaranteed as a product of the argument above due to the definition of $(\tau_m)_{m\geq m_0}$ and $\tau_{\infty}=\infty$ a.s.
\end{proof}
}

\begin{remark}
Compared to the diffusion-only Mpox model in \cite{Rahman2025}, the inclusion of Hawkes self-exciting jumps requires two additional conditions. 
First, we assume a subcritical excitation kernel
$
\int_0^\infty \nu_i(s)\,ds < 1,
$
which ensures that the intensity $\lambda_{t,i}$ remains uniformly bounded and prevents explosive jump cascades. 
Second, we impose $\varepsilon_i(y)\ge 0$, so that jumps act as nonnegative multiplicative perturbations, preserving positivity of the compartments and allowing the logarithmic Lyapunov arguments to hold.
\end{remark}

\section{Extinction of the Disease Under the Reproduction Number Threshold}

In this section we analyze the long-term behavior of the disease dynamics in relation to the basic reproduction number \( \mathcal{R}_0 \), which represents the expected number of secondary infections generated by a single infected individual in a fully susceptible population. It is as a critical threshold parameter in epidemiology. We show that when \( \mathcal{R}_0 < 1 \), the disease cannot sustain itself in the population and will eventually die out with probability one.

\begin{theorem}\label{Exti}
Let $\mathcal{X}(t)$, $t \geq 0$, be the solution to system~\eqref{System} with initial condition $\mathcal{X}(0) \in \mathbb{R}_+^6$. Assuming that (i) $\int_{\R}\epsilon_{i}^{2}(y)m(\di y)<\infty$; (ii) $\epsilon_{i}(y)>0 $;
(iii)  $\int_{0}^{+\infty} \nu_i(s) {\rm d}s<1$, then the basic reproduction number is given by 
\begin{eqnarray}\label{R}
&& \hspace*{-1.5cm}
\mathcal{R}_0 =
\frac{1}{\min\{\mu_h, \mu_r\} + \min\{\delta_h, \delta_r\}}
\nonumber
\\
&& \hspace*{-1cm}
\times
\Bigg\{
(1-p)(\eta_1 + \eta_2) + \eta_3+ 
\frac{\lambda_{0,2}\, {\mathcal G}_2}{1-\int_{0}^{+\infty} \nu_2(s)\, {\rm d}s}+\frac{\lambda_{0,3}\, {\mathcal G}_3}{1-\int_{0}^{+\infty} \nu_3(s)\, {\rm d}s}
\Bigg\},
\end{eqnarray}
where $
{\mathcal G}_i:=\int_\mathbb{R} \epsilon_i(y) m({\rm d} y),
$ for $i=2,3$.
If $\mathcal{R}_0<1$, then the number of infected individuals converges to zero almost surely, that is,
\begin{equation}\label{sum}
\lim_{t \to \infty} (I_h(t) + Q_h(t) + I_r(t)) = 0, \quad \text{a.s.}
\end{equation}
\end{theorem}

\begin{proof}
Let us apply It\^o's formula to the Lyapunov function $\mathcal{U} = \log(I_h + Q_h + I_r)$, where $\mathcal{U} : \mathbb{R}^3_+ \to \mathbb{R}$, in order to analyze the stochastic dynamics of the total number of infected individuals. The detailed computation yields the drift and diffusion components that govern the evolution of $\mathcal{U}$. Denote 
$$\Sigma(t-):=I_h(t-)+Q_h(t-)+I_r(t-),\qquad
X_4:=I_h,\ X_5:=Q_h,\ X_8:=I_r.$$
Then
\begin{align}\label{mb}
d\log\Sigma(t)
&=\frac{1}{\Sigma(t-)}\Bigg\{
\Big[(1-p)\frac{\eta_1 I_r(t)+\eta_2 I_h(t-)}{N_h(t-)}S_h(t-)-(\mu_h+\delta_h+\zeta)I_h(t-)\Big]\nonumber\\
&\qquad+\Big[\zeta I_h(t-)-(\mu_h+\gamma_h+(1-\theta)\delta_h)Q_h(t-)\Big]
+\Big[\frac{\eta_3 S_r(t)I_r(t)}{N_r(t)}-(\mu_r+\delta_r)I_r(t)\Big]\nonumber\\
&\qquad+\int_{\mathbb R}\!\big(I_h(t-)\epsilon_2(y)\lambda_2(t)+Q_h(t-)\epsilon_3(y)\lambda_3(t)\big)m(dy)
\Bigg\}dt \nonumber\\
&\quad-\frac{1}{2\Sigma(t-)^2}\Bigg\{(1-p)^2\Big(\frac{\sigma_2 S_h(t-)I_h(t-)}{N_h(t-)}\Big)^2
+\sum_{k\in\{4,5,8\}}\sigma_k^2\,X_k(t-)^2\Bigg\}dt \nonumber\\
&\quad+\frac{1}{\Sigma(t-)}\Bigg\{(1-p)\frac{\sigma_2 S_h(t-)}{N_h(t-)}I_h(t-)\,dB_2(t)
+\sum_{k\in\{4,5,8\}}\sigma_k\,X_k(t-)\,dB_k(t)\Bigg\}\nonumber\\
&\quad+\int_{\mathbb R}\log\!\Big(1+\frac{\epsilon_2(y)I_h(t-)}{\Sigma(t-)}\Big)\tilde H_2(dt,dy)
+\int_{\mathbb R}\log\!\Big(1+\frac{\epsilon_3(y)Q_h(t-)}{\Sigma(t-)}\Big)\tilde H_3(dt,dy).
\end{align}

Let us consider 
\begin{align*}
L\mathcal U(t)
&=\frac{1}{\Sigma(t-)}\Bigg\{
\Big[(1-p)\frac{\eta_1 I_r(t-)+\eta_2 I_h(t-)}{N_h(t-)}S_h(t-)-(\mu_h+\delta_h+\zeta)I_h(t-)\Big]\\
&\qquad+\Big[\zeta I_h(t-)-(\mu_h+\gamma_h+(1-\theta)\delta_h)Q_h(t-)\Big]
+\Big[\frac{\eta_3 S_r(t)I_r(t-)}{N_r(t)}-(\mu_r+\delta_r)I_r(t-)\Big]\\
&\qquad+\int_{\mathbb R}\!\big(I_h(t-)\epsilon_2(y)\lambda_2(t)+Q_h(t-)\epsilon_3(y)\lambda_3(t)\big)\,m(dy)
\Bigg\}\\
&\quad-\frac{1}{2\Sigma(t-)^2}\Bigg\{(1-p)^2\Big(\frac{\sigma_2 S_h(t-)I_h(t-)}{N_h(t-)}\Big)^2
+\sum_{k\in\{4,5,8\}}\sigma_k^2\,X_k(t-)^2\Bigg\}.
\end{align*}

\noindent
Since $ S_h \leq N_h $, $ S_r \leq N_r $ and $ Q_h \leq N_h$, we obtain the following upper bound

\begin{eqnarray*}
&&\hspace*{-1cm}
\mathcal{LU}(t) \leq 
\frac{
(1 - p)(\eta_1 + \eta_2)[I_h(t-) + I_r(t-)] + \zeta I_h(t-) +\eta_3 I_r(t-)
}{I_h(t-) + Q_h(t-) + I_r(t-)}
\\
&&\hspace*{-1cm}
-\frac{
(\mu_h + \delta_h + \zeta) I_h(t-)
+ [\mu_h + \gamma_h + (1 - \theta)\delta_h] Q_h(t-)
+ (\mu_r + \delta_r) I_r(t-)}{I_h(t-) + Q_h(t-) + I_r(t-)}
\\
&&\hspace*{-1cm}
+\frac{1}{I_h(t-) + Q_h(t-) + I_r(t-)}
\int_{\mathbb{R}} \left[ I_h(t-) \epsilon_2(y) \lambda_2(t)+ Q_h(t-) \epsilon_3(y) \lambda_3(t) \right] m(dy)
\\
&&\hspace*{-1cm}
-\frac{
(1 - p)^2 \sigma_2^2 I_h(t-)^2
+ \sigma_4^2 I_h(t-)^2 + \sigma_5^2 Q_h(t-)^2 + \sigma_8^2 I_r(t-)^2}{2(I_h(t-) + Q_h(t-) + I_r(t-))^2}.
\end{eqnarray*}
\noindent
Since the $\sigma^2$ terms are always non-negative, removing them yields the following simpler inequality
\begin{align*}
\mathcal{LU}(t) \leq \;&
(1 - p)(\eta_1 + \eta_2) + \eta_3
- \left( \min\{\mu_h, \mu_r\} + \min\{\delta_h, \delta_r\} \right)
\\
&+
\frac{\lambda_2(t) I_h(t-) \int_{\mathbb{R}} |\epsilon_2(y)| m({\rm d}y)+\lambda_3(t) Q_h(t-) \int_{\mathbb{R}} |\epsilon_3(y)| m({\rm d}y)}{I_h(t-) + Q_h(t-) + I_r(t-)}.
\end{align*}
\noindent
Thus, we obtain the following bound

\begin{align*}
\frac{1}{t} \int_{0}^t
\mathcal{LU}(s) {\rm d}s \leq \;&
(1 - p)(\eta_1 + \eta_2) + \eta_3
- \left( \min\{\mu_h, \mu_r\} + \min\{\delta_h, \delta_r\} \right) \\
&+ \frac{\Lambda_2(t)}{t} \int_{\mathbb{R}} |\epsilon_2(y)| m({\rm d}y)
+ \frac{\Lambda_3(t)}{t} \int_{\mathbb{R}} |\epsilon_3(y)| m({\rm d}y),
\end{align*}
where
$
\Lambda_i(t):=\int_{0}^t \lambda_i(s) {\rm d}s.
$
Recalling that, for a subcritical Hawkes process, 
we have  
\begin{equation}
\lim_{t\rightarrow +\infty} 
\frac{\Lambda_i(t)}{t}=\frac{\lambda_{0,i}}{1-\int_{0}^{+\infty} \nu_i(s) {\rm d}s}\quad  a.s,
\label{SLLNC}
\end{equation}
(see, for instance \cite{Fierro}), 
we conclude that 
$$
\frac{1}{t} \int_{0}^t
\mathcal{LU}(s) {\rm d}s \leq 
\left( \min\{\mu_h, \mu_r\} + \min\{\delta_h, \delta_r\} \right) (R_0-1).
$$
Let us now denote by 
\begin{align*} M_1(t)=&\int_{0}^t\frac{1}{I_h(r-) + Q_h(r-) + I_r(r-)} \Bigg[ 
(1-p) \frac{\sigma_2 S_h(r-)}{N_h(r-)} I_h(r-) \, dB_2(r) \\
&\quad + \sigma_4 I_h(r-) \, dB_4(r) + \sigma_5 Q_h(r-) \, dB_5(r) + \sigma_8 I_r(r-) \, dB_8(r) 
\Bigg].
\end{align*}
Since $S_h \leq N_h $, it follows that $\frac{S_h}{ N_h}\leq1$. Also, we note that $\frac{I_h}{I_h + Q_h + I_r}\leq1$, $\frac{Q_h}{I_h + Q_h + I_r}\leq1$, $\frac{I_r}{I_h + Q_h + I_r}\leq1$.
Hence, we obtain the estimate
\[ M_1(t)\leq 
(1-p)\sigma_2  B_2(t) + \sigma_4  B_4(t) + \sigma_5 \, B_5(t) + \sigma_8 \, dB_8(t). 
\]
By the strong low of large numbers (see \cite{GS}), for $ i=2,4,5,8,$  we have
\begin{equation*}
    \underset{t\to\infty}{\lim}\frac{B_i(t)}{t}=0 \,\,\,\textnormal{a.s.}
\end{equation*}
Therefore, it follows that 
    \begin{equation}\label{win2}
    \underset{t\to\infty}{\lim}\frac{M_1(t)}{t}=0 \,\,\,\textnormal{a.s.}
\end{equation}
The stochastic integrals
\begin{align*}
&M_2(t) =\int_0^t \int_{\mathbb{R}} \left[ \log\left( 1 + \frac{\epsilon_2(y) I_h(r-)}{I_h(r-) + Q_h(r-) + I_r(r-)} \right) \right] \tilde{H}_2(dr, dy),
\end{align*}
and
\begin{align*}
&M_3(t) = \int_0^t \int_{\mathbb{R}} \left[ \log\left( 1 + \frac{\epsilon_3(y) Q_h(r-)}{I_h(r-) + Q_h(r-) + I_r(r-)} \right)  \right] \tilde{H}_3(dr, dy),
\end{align*}
are square-integrable martingales (see \cite{kunita}, Chapter 2).

According to the It\^o isometry for integrals with respect to compensated random measures, we have:
\begin{equation*}
    \mathbb{E}[M_2(t)^2] = \mathbb{E}\left[ \int_0^t \int_{\mathbb{R}} \left( \log\left(1 + \epsilon_2(y) X(r)\right) \right)^2 
    \lambda_2(r) \, m({\rm d}y) \, {\rm d}r \right],
\end{equation*}
where $ X(r) := \frac{I_h(r)}{I_h(r) + Q_h(r) + I_r(r)} \leq 1 $. Using the analytical inequality
$ [\log(1 + x)]^2 \leq x^2$, for  $x>0$, we obtain
\begin{equation*}
     [\log\left(1 + \epsilon_2(y) X(r)\right)]^2 \leq  \epsilon_2(y)^2 X(r)^2 \leq \epsilon_2(y)^2,
\end{equation*}
since $ X(r) \leq 1 $.
Therefore,
\begin{equation*}
    \mathbb{E}[M_2(t)^2] \leq \Lambda_2(t) \int_{\mathbb{R}} \epsilon_2(y)^2 m({\rm d}y).
\end{equation*}
Hence, recalling Eq. (\ref{SLLNC}), It follows that
\begin{equation*}
    \mathbb{E}\left[\left( \frac{M_2(t)}{t} \right)^2\right] \leq \frac{1}{t} \frac{\Lambda_2(t)}{t} \int_{\mathbb{R}} \epsilon_2(y)^2 m({\rm d}y) \xrightarrow{t \to \infty} 0,\quad a.s.
\end{equation*}
Consequently, by Chebyshev's inequality and the strong law of large numbers for martingales (see, e.g., \cite{GS}), we conclude that
\begin{equation}\label{win3}
    \lim_{t \to \infty} \frac{M_2(t)}{t} = 0 \quad \text{a.s.}
\end{equation}
Similarly, we obtain
\begin{equation}\label{win4}
    \lim_{t \to \infty} \frac{M_3(t)}{t} = 0 \quad \text{a.s.}
\end{equation}
From Eq. (\ref{mb}), and recalling Eqs. (\ref{win2}), (\ref{win3}) and (\ref{win4}) we have

\begin{equation*}
    \underset{t\to\infty}{\lim \sup}\frac{\log(I_h(t) + Q_h(t) + I_r(t))}{t}\leq(\min\{\mu_h,\mu_r\}+\min\{\delta_h,\delta_r\} )(\mathcal{R}_0-1), \,\,\,\textnormal{a.s.}
\end{equation*}
Since $\mathcal{R}_0<1$ (by assumption), we obtain
\begin{equation*}
     \underset{t\to\infty}{\lim \sup}\frac{\log(I_h(t) + Q_h(t) + I_r(t))}{t}<0 \,\,\,\textnormal{a.s.}
\end{equation*}
leading to (\ref{sum}), which completes the proof.
\end{proof}


\begin{remark}
In the Hawkes-driven stochastic setting, the quantity $R_0$ in~\ref{R} should be interpreted as an \textit{effective reproduction threshold} rather than a classical deterministic basic reproduction number. It separates extinction ($R_0 < 1$) from persistence in the mean ($R_0 > 1$) under clustered transmission shocks.
\end{remark}



To analyse how the basic reproduction number $\mathcal{R}_0$ depends on contact rates, quarantine effectiveness, and Hawkes-type clustered transmission, we fix all demographic and epidemiological parameters at the baseline values listed in Table~\ref{model_params_main}, unless otherwise indicated.

\begin{table}[H]
  \centering
  \scriptsize
  \renewcommand{\arraystretch}{1.15}
  \begin{tabular}{llll}
    \toprule
    \textbf{Symbol} & \textbf{Description} & \textbf{Value (units)} & \textbf{Source / note} \\
    \midrule
    $\theta_h$ & Human recruitment rate 
               & $7.95\times10^{-5}~\text{day}^{-1}$ 
               & Peter et al.~\cite{s1} \\
    $\mu_h$    & Natural human death rate 
               & $4.11\times10^{-3}~\text{day}^{-1}$ 
               & Bhunu \& Mushayabasa~\cite{p1} \\
    $\delta_h$ & Disease-induced human death rate 
               & $5.48\times10^{-4}~\text{day}^{-1}$ 
               & Odom et al.~\cite{Odom} \\
    $\zeta$    & Detection/isolation rate $I_h \!\to Q_h$ 
               & $5.48\times10^{-3}~\text{day}^{-1}$ 
               & Peter et al.~\cite{s1}\\
    $\gamma_h$ & Recovery rate in quarantine / infected humans 
               & $2.27\times10^{-3}~\text{day}^{-1}$ 
               & Peter et al.~\cite{s1} \\
    $\theta$   & Treatment effectiveness in $Q_h$ 
               & \textit{Assumed} $\approx 0.8$ 
               & Policy / clinical effectiveness \\
    $p$        & Public enlightenment effectiveness 
               & \textit{Assumed} $\approx 0.30$ 
               & Calibration / awareness \\[1mm]
    \midrule
    $\eta_1$   & Rodent--human contact rate 
               & $6.85\times10^{-7}~\text{day}^{-1}$ 
               & Peter et al.~\cite{s1} \\
    $\eta_2$   & Human--human contact rate 
               & $1.64\times10^{-7}~\text{day}^{-1}$ 
               & Peter et al.~\cite{s1} \\
    $\eta_3$   & Rodent--rodent contact rate 
               & $7.40\times10^{-5}~\text{day}^{-1}$ 
               & Peter et al.~\cite{s1} \\[1mm]
    \midrule
    $\theta_r$ & Rodent recruitment rate 
               & $5.48\times10^{-4}~\text{day}^{-1}$ 
               & Peter et al.~\cite{s1} \\
    $\mu_r$    & Natural rodent death rate 
               & $5.48\times10^{-6}~\text{day}^{-1}$ 
               & Bhunu \& Mushayabasa~\cite{p1} \\
    $\delta_r$ & Disease-induced rodent death rate 
               & $1.37\times10^{-3}~\text{day}^{-1}$ 
               & Peter et al.~\cite{s1} \\
    \bottomrule
  \end{tabular}
  \caption{Baseline epidemiological parameters for mpox in African settings
  (following Peter et al.~\cite{s1} and related African studies),
  expressed in daily units ($\text{day}^{-1}$).}
    \label{model_params_main}
\end{table}

The Hawkes jump parameters used in the simulations are reported in Table~\ref{tab:baseline}. These values correspond to a moderate level of clustering and yield a baseline estimate of $\mathcal{R}_0\approx 2.3$, consistent with published mpox data for African settings.

\begin{table}[H]
  \centering
  \scriptsize
  \renewcommand{\arraystretch}{1.2}
  \begin{tabular}{cccccc}
    \toprule
    ${\mathcal G}_2$ & ${\mathcal G}_3$ & $\lambda_{0,2}$ & $\lambda_{0,3}$ & $\alpha_2$ & $\alpha_3$ \\
    \midrule
    1 & 1 & $2\times 10^{-4}$ & $2\times 10^{-4}$ & 0.20 & 0.15 \\
    \midrule
    $G_2$ exp. & $G_3$ exp. & base int.~2 & base int.~3 & mem.~2 & mem.~3 \\
    \bottomrule
  \end{tabular}
  \caption{Baseline Hawkes parameters used in the simulations (abbreviations: int.~=~intensity, mem.~=~memory). All intensities $\lambda_{0,i}$ are expressed in day$^{-1}$.}
  \label{tab:baseline}
\end{table}

In summary, demographic and epidemiological parameters are fixed at the values in Table~\ref{model_params_main}, while the Hawkes parameters are given in Table~\ref{tab:baseline}. Variations in the planes $(\eta_1+\eta_2,\delta_h)$, $(\eta_1+\eta_2,p)$, and $(\alpha_2,\alpha_3)$ are then explored to assess the sensitivity of $\mathcal{R}_0$.

\begin{figure}[H]
  \centering
  \includegraphics[width=0.9\textwidth]{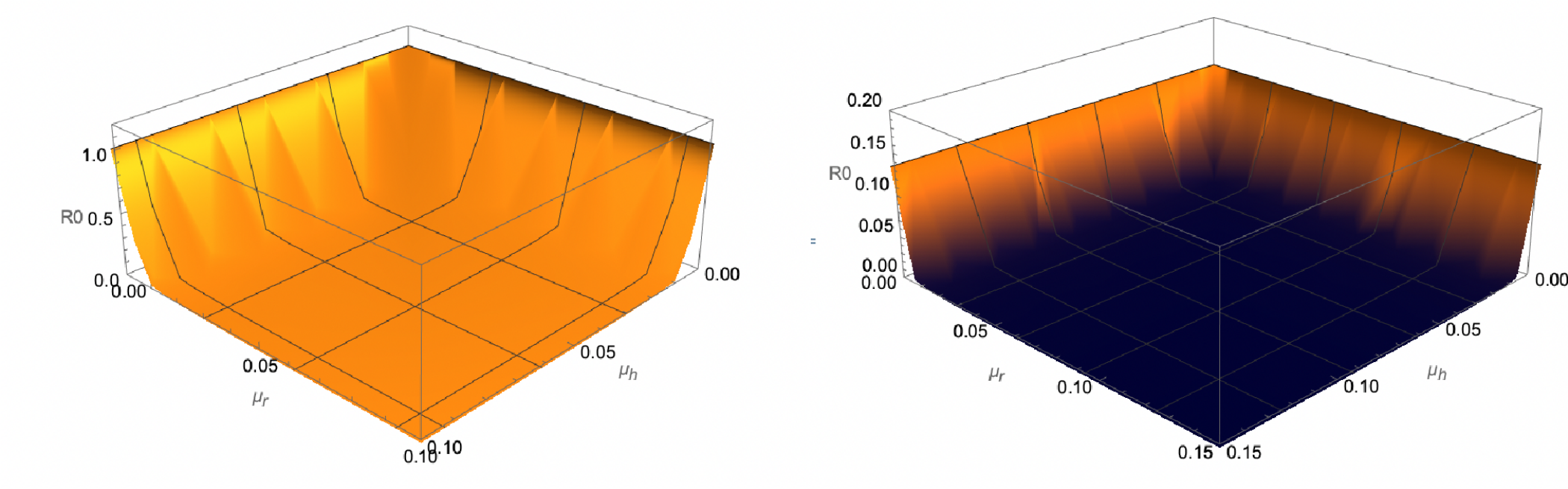}
  \caption{$\mathcal{R}_0$ as a function of the natural death rates $\mu_h$ and $\mu_r$. In the left-hand side the Hawkes jumps have exponential distribution with mean value $1$. The right-hand side describes the process in absence of jumps. The involved
  parameters are fixed according to Tables \ref{tab:baseline} and \ref{model_params_main}.}
  \label{figRomhmr}
\end{figure}

The plots in Figure \ref{figRomhmr} shows $\mathcal{R}_0$ as a function of the natural death rates in humans and rodents, both in the presence of Hawkes jumps (left-hand side) and in their absence (right-hand side).
Note that increasing the natural mortality rate 
in either humans or rodents leads to a nonlinear decrease in $\mathcal{R}_0$, and that the presence of jumps significantly increases the value of $\mathcal{R}_0$.

\begin{figure}[H]
  \centering
  \includegraphics[width=0.9\textwidth]{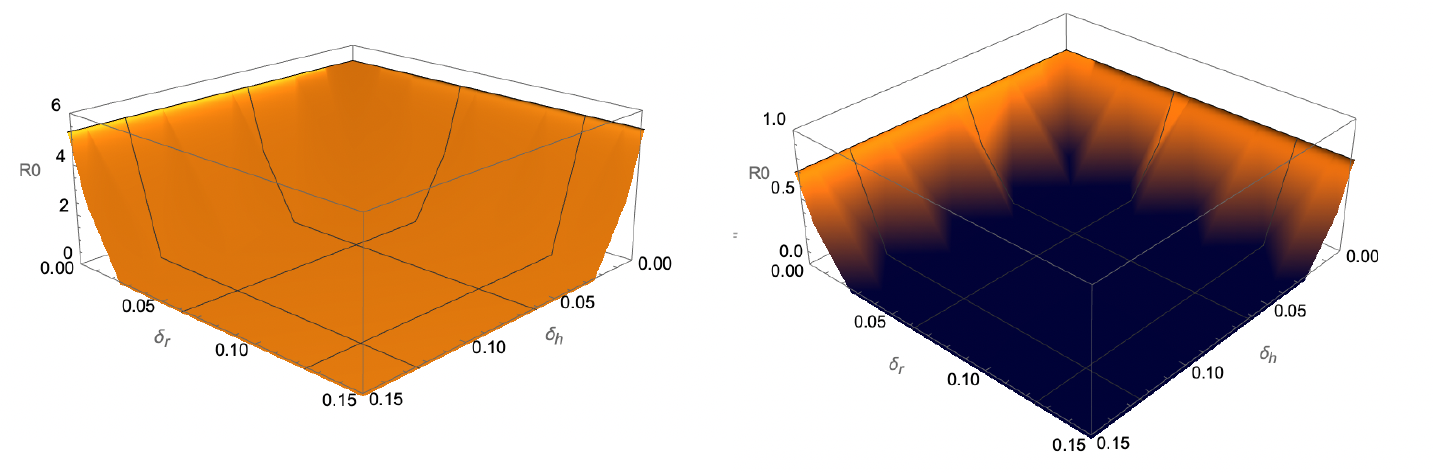}
  \caption{$\mathcal{R}_0$ as a function of the
  disease-induced death rates $\delta_h$ and $\delta_r$. 
  On the left-hand side, the Hawkes jumps are exponentially distributed with mean value $1$, 
  whereas on the right-hand side, there is absence of jumps. The involved parameters are fixed according to Tables \ref{tab:baseline} and \ref{model_params_main}.}
  \label{figRodhdr}
\end{figure}

Figure \ref{figRodhdr} shows that higher disease-induced death rates effectively reduce $\mathcal{R}_0$ by limiting the infectious period and reducing further transmission. In addition, the presence of 
exponential Hawkes jumps with mean $1$ drives $\mathcal{R}_0$ 
above $1$ for small values of $\delta_h$ or $\delta_h$. In contrast, in the absence of jumps, $\mathcal{R}_0$ consistently remains below $1$.

To further investigate control strategies, we analyse how the basic reproduction
number $\mathcal{R}_0$ depends on contact intensity, disease-induced mortality,
and the contribution of Hawkes jumps. Figure~\ref{eta-delta} reports two
heat maps computed with the empirical parameters from 
Tables~\ref{model_params_main}--\ref{tab:baseline} and the baseline Hawkes
settings.
\begin{figure}[H]
  \centering
  \includegraphics[width=0.95\textwidth]{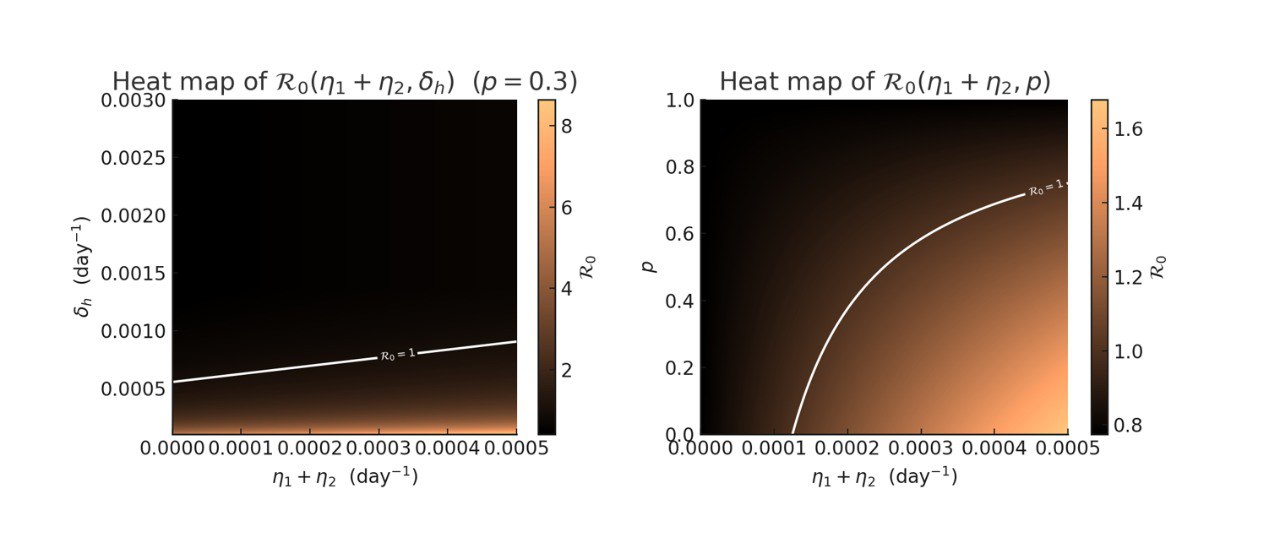}
  \caption{
Heat maps of $\mathcal{R}_0$ computed with baseline parameter values.
  (Left) Dependence on the combined rodent--human contact rate $\eta_1+\eta_2$
  and the human disease-induced mortality $\delta_h$, with public enlightenment
  fixed at $p=0.3$. (Right) Dependence on $(\eta_1+\eta_2,p)$ for fixed
  $\delta_h=10^{-3}\,\text{day}^{-1}$. The white contour marks the threshold
  $\mathcal{R}_0=1$.
}
  \label{eta-delta}
\end{figure}
In Figure~\ref{eta-delta}, higher $\delta_h$ shifts the system into the
subcritical regime even for relatively large contact intensities, while
increasing public enlightenment $p$ compensates for high transmission pressure
and expands the region where $\mathcal{R}_0<1$.

\begin{figure}[H]
  \centering
  \includegraphics[width=0.98\textwidth]{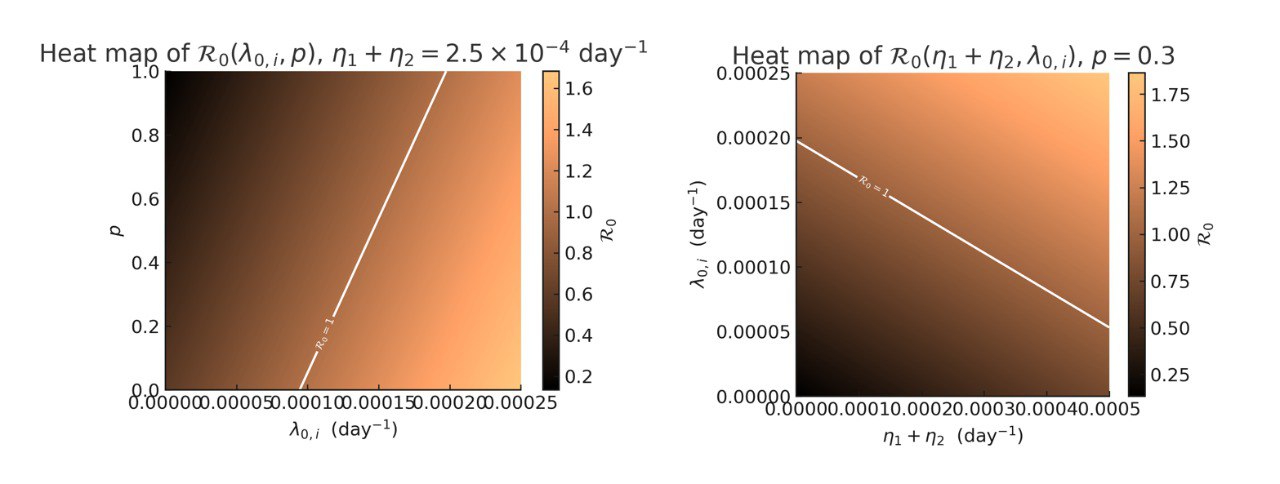}
 \caption{Effect of clustered transmission on $\mathcal{R}_0$. Left:
 $\mathcal{R}_0$ as a function of $(\lambda_{0,i},p)$ with
 $\eta_1+\eta_2=2.5\times10^{-4}\,\text{day}^{-1}$. Right:
 $\mathcal{R}_0$ as a function of $(\eta_1+\eta_2,\lambda_{0,i})$ with $p=0.3$.
 All other parameters follow Tables~\ref{model_params_main}--\ref{tab:baseline}.
 The white line indicates $\mathcal{R}_0=1$.}
  \label{fig_eta-delta}
\end{figure}

Figure~\ref{fig_eta-delta} highlights the dominant influence of clustered
transmission. Larger baseline jump intensities $\lambda_{0,i}$ require stronger
public-health interventions (e.g., higher $p$ or lower contact rates) to keep
$\mathcal{R}_0$ below the epidemic threshold. The critical contour
$\mathcal{R}_0=1$ delineates an admissible region of low contact intensity and
weak clustering from a supercritical region driven by frequent transmission
bursts.

Having specified the model and baseline parameters in
Tables~\ref{model_params_main}--\ref{tab:hawkes_params}, 
we now turn to the stochastic simulations of sample paths of system~\eqref{System1}. Numerical simulations were performed using a jump-adapted Euler-Maruyama scheme for stochastic differential equations with multiplicative jumps (see \cite{barcode})
Unless stated otherwise, all numerical experiments use this reference parameter set.
Random environmental fluctuations are represented by diffusion terms in each equation,
with volatility coefficients
$\sigma_1=\dots=\sigma_6=0.3$ for the human compartments and
$\sigma_7=\sigma_8=0.05$ for the rodent compartments.
These diffusion terms were chosen heuristically so that the simulated trajectories
exhibit a realistic level of variability around the deterministic dynamics.

Jump sizes $\epsilon_i$ are generated as exponential marks with mean $G_i$,
$
\epsilon_{i}^{\mathrm{raw}} \sim \mathrm{Exp}(\text{mean}=G_i),
$
and then truncated in order to avoid unrealistically large shocks
\[
\epsilon_i = \min\{\epsilon_{i}^{\mathrm{raw}}, \epsilon_{\max}\}, \qquad \epsilon_{\max}=3,
\]
so that a single jump can increase a compartment by at most a factor $1+\epsilon_{\max}=4$. This assumption is introduced purely for simulation stability and does not enter the theoretical analysis. All analytical results require only positivity and finite second moments of $\epsilon_i$, which hold irrespective of this truncation.
The corresponding probability measure $m(dy)$ is the law of this truncated exponential
distribution, which still satisfies $m(\{0\})=0$ and has finite second moments.

In the numerical scheme, jumps act multiplicatively on the affected compartment,
$
X \leftarrow X \prod_{j=1}^{N_{i,t}} (1+\epsilon_{i,j}),
$
where $X \in \{S_h,I_h,Q_h,R_h\}$, $N_{i,t}$ is the number of jumps in channel $i$
during the current time step, and $\epsilon_{i,j}$ are the truncated marks. This
update interprets jumps as relative increases, preserves positivity of all state variables
and, together with truncation, keeps trajectories within biologically realistic ranges.

\begin{table}[H]
  \centering
  \scriptsize
  \renewcommand{\arraystretch}{1.15}
  \begin{tabular}{llll}
    \toprule
    \textbf{Parameter} & \textbf{Description} & \textbf{Value} & \textbf{Source / note} \\
    \midrule
    $\lambda_{0,i}$ & Baseline intensity (for $S_h,I_h,Q_h,R_h$) & $2\times 10^{-4}$ & \emph{Assumed} \\
    $\alpha_i$ & Excitation strength & $0.2$ ($i=1,2$), $0.15$ ($i=3,4$) & \emph{Assumed} \\
    $\beta_i$  & Memory decay rate & $1.0$ & \emph{Assumed} \\
    $\mathcal{G}_i$ & Mean jump size (exp. marks) & $1$ & \emph{Assumed} \\
    \bottomrule
  \end{tabular}
  \caption{Hawkes process parameters used for human compartments. All values are assumed for illustration of burst-like epidemic shocks.}
  \label{tab:hawkes_params}
\end{table} 
Tables~\ref{model_params_main}--\ref{tab:hawkes_params} show all parameters used in the simulations. 
We ran 80 stochastic scenarios for each variable. 
Thin lines represent single trajectories, the thick line shows the average, and dots mark Hawkes jumps in human compartments. 
\begin{figure}[H]
  \centering
  \includegraphics[width=0.95\textwidth]{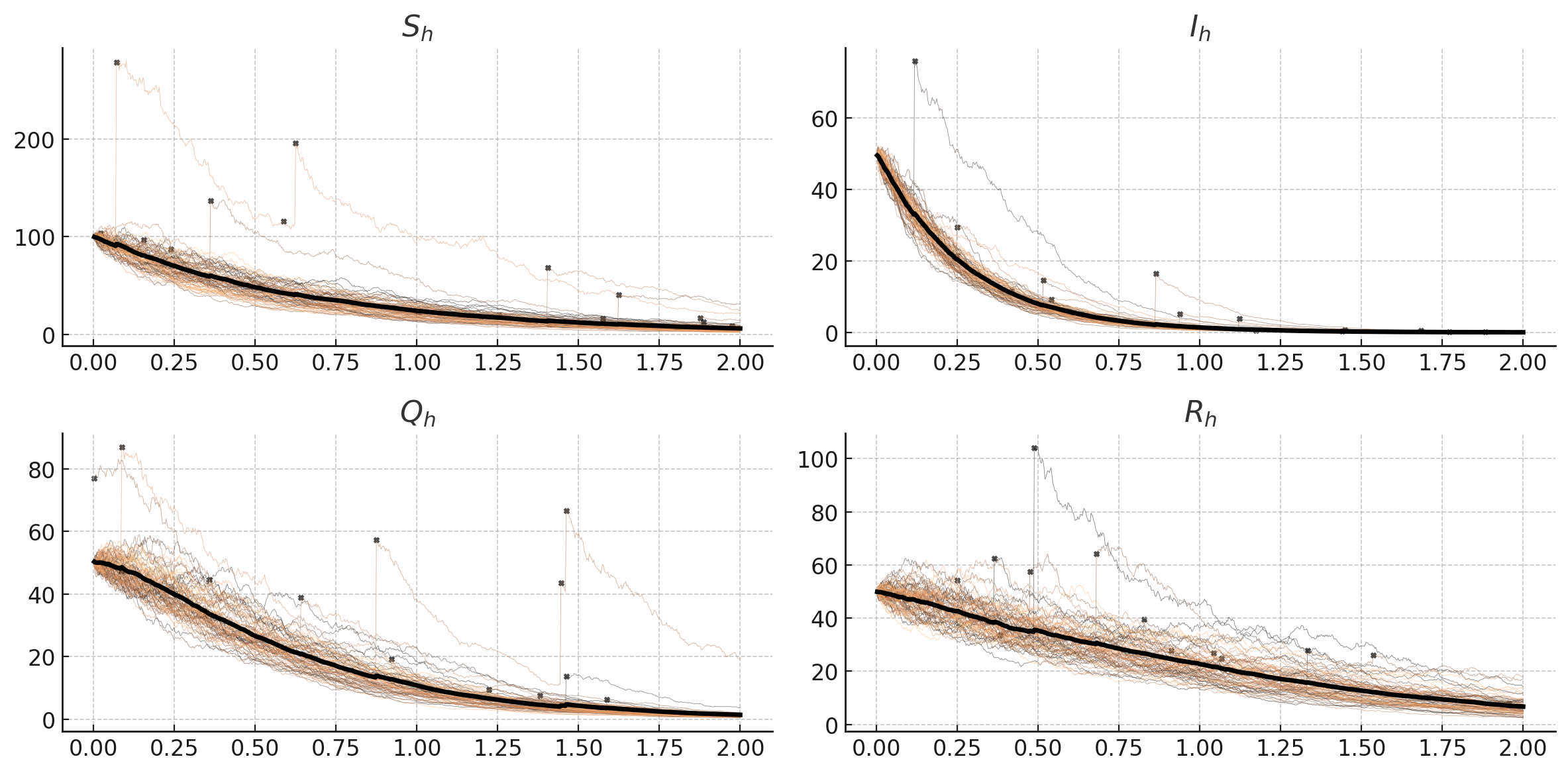}
  \caption{Stochastic simulations of the human compartments. 
  Each subplot shows 80 trajectories (thin lines), the mean path (thick curve), 
  and Hawkes jumps. Simulations were run with the baseline parameters 
  in Tables~\ref{model_params_main}--\ref{tab:hawkes_params}. 
  Diffusion noise and Hawkes jumps were included. }
  \label{humans}
\end{figure}
Figure~\ref{humans} presents the four human compartments: susceptible $S_h$, infected $I_h$, quarantined $Q_h$, and recovered $R_h$. 
The results illustrate how diffusion noise and Hawkes jumps create variation between runs and produce sudden epidemic shocks.

Figure~\ref{rodents} shows stochastic sample paths for the rodent population: susceptible rodents $S_r$ and infected rodents $I_r$. We simulated 80 trajectories for each compartment using the same baseline parameters as in the human case; thin lines denote individual paths and the thick curve their mean. Unlike for humans, only diffusion noise was included for rodents, with no Hawkes jumps.

\begin{figure}[H]
  \centering
  \includegraphics[width=0.82\textwidth]{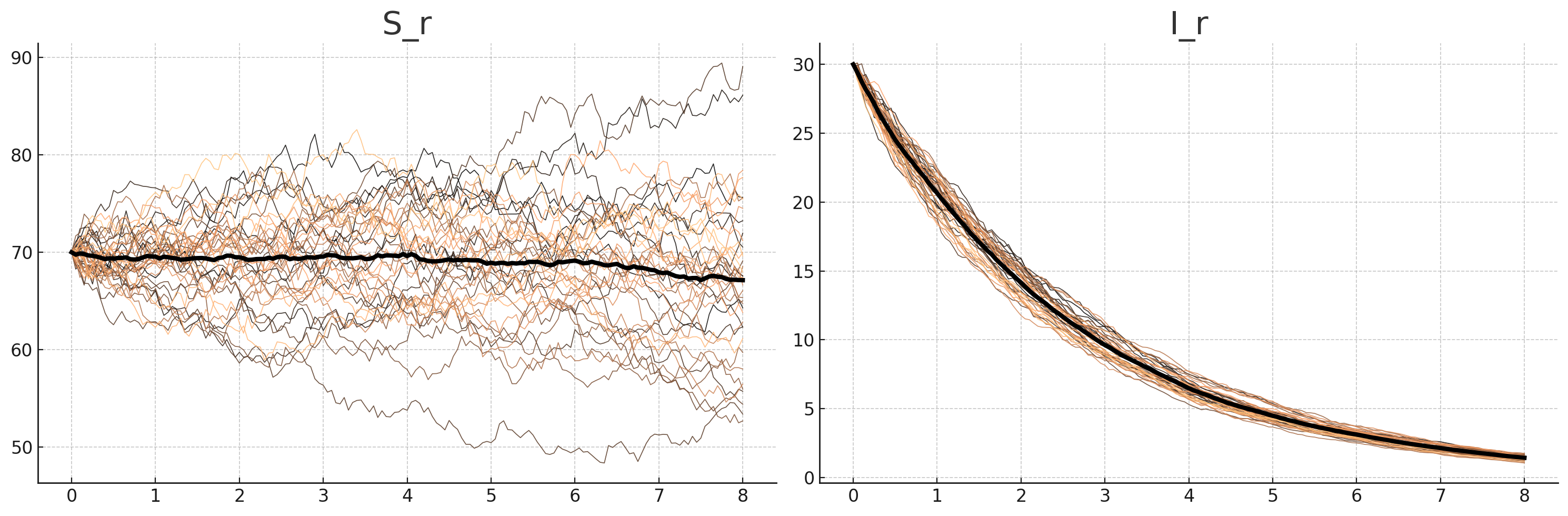}
  \caption{Stochastic simulations of the rodent compartments. 
  Each subplot shows 80 trajectories (thin lines) and the mean path (thick curve). 
  Only diffusion noise was added, because Hawkes jumps were applied only to the human compartments. 
 }
  \label{rodents}
\end{figure}

\section{Persistence Thresholds for the Coupled Human-Rodent Dynamics}
In this section we explain what long-term survival of the disease means in our human-rodent model.
\begin{definition}
    The system \eqref{System1} is said to be persistent in the mean if the following condition holds:
  \begin{equation}\label{ir}
\liminf_{t \to \infty} \frac{1}{t} \int_0^t I_r(s)ds> 0 \quad \text{a.s.}
\end{equation}
and 
 \begin{equation}\label{ih}
     \liminf_{t \to \infty} \frac{1}{t} \int_0^t I_h(s)ds> 0 \quad \text{a.s.}
 \end{equation}

\end{definition} 
Note that the condition of persistence in the mean (positive stability) primarily concerns the long-term survival of the infection in the system. Since individuals in quarantine $Q_h$ do not spread the infection, their presence in the long run does not determine whether the disease will remain in the population. The key factor for infection persistence is whether the population of infected individuals $I_h$ and infected rodents $ I_r$ remains strictly positive on average over time. If infected humans $ I_h $  is the only class being evaluated for stability, this may be insufficient to assess the survival of the virus. It is also necessary to consider infected rodents $I_r$, as they may sustain viral circulation even when human cases decline.
\begin{lemma}\label{Lemma5.1.} Let $f \in \mathcal{C}[[0,\infty) \times \Omega, (0,\infty)]$. If there exist positive constants $\lambda_0, \lambda$ such that
\[
\log f(t) \geq \lambda t - \lambda_0 \int_0^t f(s) \, ds + F(t), \quad \text{a.s.}
\]
\textit{for all $t \geq 0$, where $F \in \mathcal{C}[[0,\infty) \times \Omega, \mathbb{R}]$ and $\underset{t \to \infty}\lim \frac{F(t)}{t} = 0$ a.s., then}
\[
\liminf_{t \to \infty} \frac{1}{t} \int_0^t f(s) \, ds \geq \frac{\lambda}{\lambda_0} \quad \text{a.s.}
\]
\end{lemma}
We first prove  auxiliary results that will be needed to identify the persistent conditions.

\begin{theorem}[Rodent persistence]\label{thm:rodent-persist-noeps}
Assume positive constants $\mu_r$, $\delta_r$, $\eta_3$, $\sigma_8$ and that $N_r(t)\ge \underline N_r>0$ for all $t\ge0$ a.s. If
$
a:=\eta_3-\mu_r-\delta_r-\tfrac12\,\sigma_8^2>0,
$
then
\[
\liminf_{t\to\infty}\frac1t\int_0^t I_r(s)\,ds\ \ge\
\frac{\underline N_r}{\eta_3}\,\Big(\eta_3-\mu_r-\delta_r-\tfrac12\sigma_8^2\Big)\ >\ 0\quad\text{a.s.}
\]
\end{theorem}

\begin{proof}
It\^o for $\ln I_r$ gives
\[
d\ln I_r=\Big[\eta_3\,\tfrac{S_r}{N_r}-\mu_r-\delta_r-\tfrac12\sigma_8^2\Big]dt+\sigma_8\,dB_8.
\]
Since $N_r=S_r+I_r$, $\tfrac{S_r}{N_r}=1-\tfrac{I_r}{N_r}$, it results
\[
\ln I_r(t)=\ln I_r(0)+a t-\eta_3\!\int_0^t\!\frac{I_r}{N_r}\,ds+\int_0^t\sigma_8\,dB_8.
\]
Since $N_r\ge \underline N_r$, we have $-\eta_3\!\int_0^t\!\frac{I_r}{N_r}\,ds\ \ge\ -\frac{\eta_3}{\underline N_r}\!\int_0^t\! I_r(s)\,ds$.
Thus
\[
\ln I_r(t)\ \ge\ C + a t - b \int_0^t I_r(s)\,ds + M_t,
\]
with $b=\eta_3/\underline N_r$ and $M_t=\int_0^t\sigma_8\,dB_8(s)$ satisfying $\underset{t\to\infty}\lim\frac{M_t}{t}=0$ a.s. The standard logarithmic Gronwall-type lemma for semimartingales then yields
$$
\liminf_{t\to\infty}\frac1t\int_0^t I_r(s)\,ds\ \ge\ \frac{a}{b}.
$$
\end{proof}

For the human subsystem, we begin with an eventual lower bound on the susceptible fraction. This mild assumption captures regimes where control measures or demography prevent depletion of susceptibles over long horizons.

\begin{lemma}[Availability of susceptibles from the model equations]\label{avSh} Consider the human subsystem in the compensated Hawkes formulation \eqref{System1}. Assume the standing bounds \eqref{BTP}- \eqref{HR} hold. Let all human parameters $\theta_h,\mu_h,\delta_h,\zeta,\eta_1,\eta_2,p$ and diffusion coefficients $\sigma_1,\sigma_2,\sigma_3$ be constant and nonnegative. Define \[ K^*:=\limsup_{T\to\infty}\ \sup_{0\le s\le T}k(s)\ \in(0,+\infty]\!. \] Then we have the time-average lower bound \begin{equation}\label{intShNhT0} \liminf_{T\to\infty}\ \frac1T\int_0^T\frac{S_h(s)}{N_h(s)}\,ds \ \ \ge\ \ \frac{\theta_h}{\,M\big(\mu_h+(1-p)\big(\eta_2+\eta_1K^*\big)\big)}\,:=\varepsilon_h. \end{equation} In particular, if $K^*<\infty$ then the right-hand side is strictly positive. \end{lemma} \begin{proof} From the $S_h$-equation in \textup{\eqref{System1}}, we have, a.s., \[ \begin{aligned} d S_{h}(t)=&\left(\theta_{h}-(1-p)\cdot\frac{\eta_1I_r(t-)+\eta_2I_h(t-)}{N_h(t-)}S_{h}(t-)-\mu_{h}S_{h}(t-)\right)dt\notag\\ &-(1-p)\dfrac{\sigma_{1}I_{r}(t-)d B_{1}(t)+\sigma_{2}I_{h}(t-)dB_{2}(t)}{N_{h}(t-)}S_{h}(t-)+\sigma_{3}S_{h}(t-)dB_{3}(t)\notag\\ &+\left(\int_{\mathbb{R}}S_{h}(t-)\epsilon_{1}(y)\lambda_{t,1}m(dy)\right)dt
+\int_{\mathbb{R}}\epsilon_{1}(y)S_{h}(t-)\tilde{H}_{1}(dt,dy).
\end{aligned} \] 
Let us integrate from $0$ to $T$, divide by $T$ and rearrange. Using $S_h(T)-S_h(0)=o(T)$ and the strong law for the martingale parts, we obtain \[ \theta_h = \mu_h\,\frac1T\!\int_0^T\! S_h\,ds +(1-p)\,\frac1T\!\int_0^T\!\frac{\eta_1 I_r+\eta_2 I_h}{N_h}\,S_h\,ds -\frac1T\!\int_0^T\!\Big(\!\int_{\mathbb R}\! S_h\,\epsilon_1\lambda_{s,1}m(dy)\Big)ds + o(1). \] Discarding the last (nonnegative) term to get a conservative inequality and bounding the incidence 
using \eqref{HR} and $I_h\le N_h$, we obatin \[ \frac{\eta_1 I_r+\eta_2 I_h}{N_h}\ \le\ \eta_1\,k(s)+\eta_2\ \le\ \eta_1 K_T+\eta_2, \qquad K_T:=\sup_{0\le s\le T}k(s). \] Hence, \[ \theta_h \ \le\ \big(\mu_h+(1-p)(\eta_2+\eta_1K_T)\big)\ \frac1T\!\int_0^T\! S_h(s)\,ds\ +\ o(1). \] Letting $T\to\infty$ and using $K^*=\underset{T\to\infty}{\lim \sup} K_T$, we get \[ \liminf_{T\to\infty}\frac1T\!\int_0^T\! S_h(s)\,ds \ \ \ge\ \ \frac{\theta_h}{\,\mu_h+(1-p)\big(\eta_2+\eta_1K^*\big)}. \] Finally, $N_h\le M$ by \eqref{BTP}, so $S_h/N_h \ge S_h/M$, which yields \[ \liminf_{T\to\infty}\frac1T\!\int_0^T\!\frac{S_h(s)}{N_h(s)}\,ds \ \ge\ \frac1M\ \liminf_{T\to\infty}\frac1T\!\int_0^T\! S_h(s)\,ds \ \ge\ \frac{\theta_h}{\,M\big(\mu_h+(1-p)\big(\eta_2+\eta_1K^*\big)\big)}. \] \end{proof} \begin{remark} The positive compensated Hawkes drift in $dS_h$ (the $\int S_h\,\epsilon_1\lambda_{t,1}m(dy)$ term) was dropped above for a conservative estimate. Keeping it and using $S_h\le N_h\le M$ together with the law of large numbers for Hawkes intensities yields an additive improvement proportional to the stationary mean of $\lambda_{t,1}$. \end{remark} Next, we derive a refined lower bound for the time integral of $S_h/N_h$ by applying the ratio It\^o's formula with Hawkes jump terms. Importantly, the estimate keeps only the negative drift that must be controlled and the positive contributions depending solely on $S_h/N_h$, without assuming a lower bound on $I_h$.

\begin{lemma}[Ratio estimate under a time-averaged susceptible lower bound]\label{taverage} Assume the compensated human subsystem \eqref{System1} with positive constant coefficients. Let the Hawkes channels $i=1,2,3,4$ admit predictable intensities $\lambda_{t,i}$ and suppose $\mathcal{G}_i := \int_{\mathbb R} \epsilon_i(y)\,m(dy)\in(0,\infty)$ with $\epsilon_i(y)\ge0$. Let $\epsilon_h\in(0,1]$ be the constant furnished by Lemma~\ref{avSh} and fix any $\delta\in(0,\epsilon_h)$. Set $a:=\epsilon_h-\delta>0$. Then there exist $T_1\ge T_0$, a finite a.s. constant $\widetilde C_0$, and a square-integrable martingale $\widetilde M_t$ such that for all $t\ge T_1$, a.s.,
\begin{align}\label{average} \int_{T_0}^{t} \frac{S_h(s)}{N_h(s)}\,ds\nonumber& \ge\widetilde C_0 +\frac{1}{\mu_h}\Big( \mu_h a^2 + (a^3-1)\,\sigma_3^2 \Big)(t-T_0) - \frac{\eta_1+\eta_2}{\mu_h}\int_{T_0}^{t} \frac{I_h(s)}{N_h(s)}\,ds\nonumber\\ & - \frac{1}{\mu_h}\sum_{i=2}^4 \mathcal{G}_i\int_{T_0}^{t} \lambda_{s,i}\,ds + \widetilde M_t.
\end{align}
\end{lemma}
\begin{proof} By Lemma~\ref{avSh}, for the chosen $\delta\in(0,\epsilon_h)$ there exist $T_1\ge T_0$ and $C_\delta<\infty$ (a.s.) such that, for all $t\ge T_1$, 
\begin{equation}\label{eq:avg-ShNh} \int_{T_0}^{t} \frac{S_h}{N_h}\,ds \;\ge\; a\,(t-T_0)\; -\; C_\delta,\qquad a:=\varepsilon_h-\delta>0. 
\end{equation}
Since $N_h=S_h+I_h+Q_h+R_h$, then
\begin{align} dN_h(t) &=\Bigg[ \theta_h -\mu_h\,N_h(t-) -\delta_h\,I_h(t-) -(1-\theta)\,\delta_h\,Q_h(t-) \notag\\ 
&\qquad +\int_{\mathbb R}\!\epsilon_{1}(y)\,S_h(t-)\,\lambda_{t,1}\,m(dy) +\int_{\mathbb R}\!\epsilon_{2}(y)\,I_h(t-)\,\lambda_{t,2}\,m(dy) \notag\\
&\qquad +\int_{\mathbb R}\!\epsilon_{3}(y)\,Q_h(t-)\,\lambda_{t,3}\,m(dy) +\int_{\mathbb R}\!\epsilon_{4}(y)\,R_h(t-)\,\lambda_{t,4}\,m(dy) \Bigg]dt \notag\\[2mm] 
&\quad -(1-p)\,\frac{\sigma_{1}\,I_r(t-)}{N_h(t-)}\,S_h(t-)\,dB_1(t) +\sigma_{3}\,S_h(t-)\,dB_3(t) +\sigma_{4}\,I_h(t-)\,dB_4(t) \notag\\
&\quad +\sigma_{5}\,Q_h(t-)\,dB_5(t) +\sigma_{6}\,R_h(t-)\,dB_6(t) \notag\\[2mm]
&\quad +\int_{\mathbb R}\!\epsilon_{1}(y)\,S_h(t-)\,\tilde H_1(dt,dy) +\int_{\mathbb R}\!\epsilon_{2}(y)\,I_h(t-)\,\tilde H_2(dt,dy) \notag\\
&\quad +\int_{\mathbb R}\!\epsilon_{3}(y)\,Q_h(t-)\,\tilde H_3(dt,dy) +\int_{\mathbb R}\!\epsilon_{4}(y)\,R_h(t-)\,\tilde H_4(dt,dy).\label{eq27} 
\end{align} 

Apply It\^o's formula for jump-diffusions in the Hawkes-It\^o form (see Eq.~(\ref{Ito1}) in the It\^o formula for Hawkes-It\^o processes) to
$f(x,y)=x/y$ with $x=S_h$ and $y=N_h$ on $[T_0,t]$.
For $f(x,y)=x/y$ one has
\[
f_x=\frac{1}{y},\qquad f_y=-\frac{x}{y^2},\qquad
f_{xx}=0,\qquad f_{xy}=-\frac{1}{y^2},\qquad f_{yy}=\frac{2x}{y^3}.
\]
Hence,
\begin{align}
&d\!\left(\frac{S_h}{N_h}\right)
= \frac{1}{N_h}\,dS_h \;-\; \frac{S_h}{N_h^{2}}\,dN_h
\;+\;\frac12\,f_{yy}(S_h,N_h)\,d\langle N_h^c\rangle_t
\;+\;f_{xy}(S_h,N_h)\,d\langle S_h^c,N_h^c\rangle_t  \nonumber\\
&\quad + \sum_{i=1}^{4}\int_{\mathbb R}
\Bigg[
f\big(S_h(t-)+\Delta_i^{S}(t,y),\,N_h(t-)+\Delta_i^{N}(t,y)\big)
- f\big(S_h(t-),N_h(t-)\big)
\Bigg]\,H_i(dt,dy),\label{eq:ratio-ito}
\end{align}
where $N_h^c$ and $S_h^c$ denote the continuous martingale parts and the jump increments
$(\Delta_i^{S},\Delta_i^{N})$ are given by
\[
(\Delta_i^{S},\Delta_i^{N})=
\begin{cases}
\big(\epsilon_1(y)S_h,\ \epsilon_1(y)S_h\big), & i=1,\\[2pt]
\big(0,\ \epsilon_i(y)X_i\big), & i=2,3,4,
\end{cases}
\qquad
X_2=I_h,\ X_3=Q_h,\ X_4=R_h.
\]
Let us decompose $H_i(dt,dy)=\lambda_{t,i}m(dy)\,dt+\widetilde H_i(dt,dy)$ and set
\[
\Delta^{(i)}_f(t,y):=
f\big(S_h(t-)+\Delta_i^{S}(t,y),\,N_h(t-)+\Delta_i^{N}(t,y)\big)
- f\big(S_h(t-),N_h(t-)\big).
\]
Then the jump contribution in \eqref{eq:ratio-ito} equals
\[
\sum_{i=1}^4\int_{\R}\Delta^{(i)}_f(t,y)\,\lambda_{t,i}m(dy)\,dt
+\sum_{i=1}^4\int_{\R}\Delta^{(i)}_f(t,y)\,\widetilde H_i(dt,dy).
\]
Moreover, writing $r(t):=\frac{S_h(t)}{N_h(t)}\in[0,1]$ and $u_i(t):=\frac{X_i(t)}{N_h(t)}\in[0,1]$, one checks that
\[
\Delta^{(1)}_f(t,y)
=\frac{S_h(1+\epsilon_1(y))}{N_h+\epsilon_1(y)S_h}-\frac{S_h}{N_h}
= r(t-)\,\frac{\epsilon_1(y)\big(1-r(t-)\big)}{1+\epsilon_1(y)r(t-)}\ \ge\ 0,
\]
and for $i=2,3,4$,
\[
\Delta^{(i)}_f(t,y)
=\frac{S_h}{N_h+\epsilon_i(y)X_i}-\frac{S_h}{N_h}
= -\,r(t-)\,\frac{\epsilon_i(y)u_i(t-)}{1+\epsilon_i(y)u_i(t-)}
\ \ge\ -\,\epsilon_i(y).
\]

Collecting the  drift terms in \eqref{eq:ratio-ito} and keeping only the negative contributions that are decisive for a lower bound,
we obtain
\[
\begin{aligned}
d\!\left(\frac{S_h(t)}{N_h(t)}\right)
&\ge \Biggl[
-\mu_h \frac{S_h(t)}{N_h(t)}
-(\eta_1+\eta_2)\frac{I_h(t)}{N_h(t)}
+\sigma_3^{2}\!\left( \left(\frac{S_h(t)}{N_h(t)}\right)^{3} -\left(\frac{S_h(t)}{N_h(t)}\right)^{2} \right)
\Biggr]dt \\
&\quad +\sum_{i=1}^{4}\int_{\mathbb R}\Delta^{(i)}_f(t,y)\,\lambda_{t,i}m(dy)\,dt
+\sum_{i=1}^{4}\int_{\mathbb R}\Delta^{(i)}_f(t,y)\,\widetilde H_i(dt,dy)
\;+\; M_{B}(t),
\end{aligned}
\]
where $M_B(t)$ collects the Brownian local martingale terms.
Integrating over $[T_0,t]$, dividing by $\mu_h>0$ and gathering bounded initial terms into $C_0$, we get
\begin{align}
\int_{T_0}^{t} \frac{S_h}{N_h}\,ds
&\ge C_0
-\frac{\eta_1+\eta_2}{\mu_h}\int_{T_0}^{t}\frac{I_h}{N_h}\,ds
+\frac{\sigma_3^2}{\mu_h}\int_{T_0}^{t}\Big[\Big(\frac{S_h}{N_h}\Big)^3-\Big(\frac{S_h}{N_h}\Big)^2\Big]ds
+\frac{1}{\mu_h}\int_{T_0}^{t}\mu_h\Big(\frac{S_h}{N_h}\Big)^2ds \nonumber\\
&\quad+\frac{1}{\mu_h}\sum_{i=1}^4\int_{T_0}^{t}\!\int_{\R}\Delta^{(i)}_f(s,y)\,\lambda_{s,i}m(dy)\,ds
+\text{martingale}.\label{eq:ratio-raw}
\end{align}
Using the bounds on the true post-jump ratio increments derived above, we have
\begin{equation}\label{eq:jump-bounds}
\int_{\mathbb R}\Delta^{(1)}_f(s,y)\,\lambda_{s,1}m(dy)\ \ge\ 0,\qquad
\int_{\mathbb R}\Delta^{(i)}_f(s,y)\,\lambda_{s,i}m(dy)\ \ge\ -\mathcal G_i\,\lambda_{s,i},\ \ i=2,3,4.
\end{equation}
  By Cauchy-Schwarz, \[ \int_{T_0}^{t} \Big(\frac{S_h(s)}{N_h(s)}\Big)^2 ds\;\ge\; \frac{\big(\int_{T_0}^{t} S_h(s)/N_h(s) ds\big)^2}{t-T_0}. \] By Jensen's inequality for the convex function  on $[0,1]$, \[ \int_{T_0}^{t} \Big(\frac{S_h(s)}{N_h(s)}\Big)^3 ds\;\ge\; (t-T_0)\Big(\frac{1}{t-T_0}\int_{T_0}^{t} \frac{S_h(s)}{N_h(s)} ds\Big)^3. \] Using \eqref{eq:avg-ShNh}, both bounds are linear from below in $(t-T_0)$, namely 
  \begin{equation}\label{eq:power-bounds} \int_{T_0}^{t} \Big(\frac{S_h(s)}{N_h(s)}\Big)^2 ds\;\ge\; a^2 (t-T_0) - C_1,\qquad \int_{T_0}^{t} \Big(\frac{S_h(s)}{N_h(s)}\Big)^3 ds\;\ge\; a^3 (t-T_0) - C_2, 
  \end{equation}
  for some finite a.s. constants $C_1$, $C_2$ (absorbed into $\widetilde C_0$ below). Substitute \eqref{eq:jump-bounds}-\eqref{eq:power-bounds} into \eqref{eq:ratio-raw}, absorb finite a.s. constants into $\widetilde C_0$, and gather the local martingales into $\widetilde M_t$ (square-integrable on finite horizons). This gives exactly \eqref{average}. 
  \end{proof}
  
\begin{theorem}[Human persistence in the mean, constant coefficients]
\label{thm:human-persist-const}
Assume the compensated system \eqref{System1} with constant parameters
$\mu_h>0$, $\delta_h,\zeta,\eta_1,\eta_2,p\ge0$
and volatilities $\sigma_2,\sigma_3,\sigma_4\ge0$.
Assume the standing bounds \eqref{BTP}--\eqref{HR} and that there exists
$\underline N_h>0$ such that $N_h(t)\ge \underline N_h$ a.s.\ for all $t\ge T_0$.
Let the Hawkes channels $i=1,2,3,4$ admit almost surely stationary time averages
$\bar\lambda_i$ and define
\[
{\mathcal G}_i:=\int_{\R}\epsilon_i(y)\,m(dy)\in(0,\infty),
\qquad
L_2:=\int_{\R}\ln(1+\epsilon_2(y))\,m(dy)\in(0,\infty).
\]
Let $\epsilon_h\in(0,1]$ be the constant provided by Lemma~\ref{avSh}.
If
\begin{equation}\label{eq:lambda-h-new}
\begin{aligned}
\lambda_h :=\ &
(1-p)\eta_2
\Bigg[
\frac{\epsilon_h^2\mu_h+(\epsilon_h^3-1)\sigma_3^2}{\mu_h}
-\frac{1}{\mu_h}\sum_{i=2}^4 {\mathcal G}_i\,\bar\lambda_i
\Bigg]
\\[-1mm]
&\quad
-\Big(\mu_h+\delta_h+\zeta
+\tfrac12\big(\sigma_4^2+(1-p)^2\sigma_2^2\big)\Big)
+L_2\,\bar\lambda_2
\;>\;0,
\end{aligned}
\end{equation}
and
\[
\lambda_{0,h}:=\frac{(1-p)\eta_2(\eta_1+\eta_2)}{\mu_h\,\underline N_h},
\]
then the human infection is persistent in the mean, namely
\[
\liminf_{t\to\infty}\frac1t\int_0^t I_h(s)\,ds
\ \ge\ \frac{\lambda_h}{\lambda_{0,h}}
\ >\ 0
\qquad\text{a.s.}
\]
\end{theorem}

\begin{proof}
Consider the jump--diffusion dynamics for $I_h$:
\[
\begin{aligned}
dI_h(t)
&=\Bigg[
(1-p)\frac{\eta_1 I_r(t-)+\eta_2 I_h(t-)}{N_h(t-)}S_h(t-)
-(\mu_h+\delta_h+\zeta)I_h(t-)
\\
&\qquad\qquad
+\int_{\R}\epsilon_2(y)I_h(t-)\lambda_{t,2}m(dy)
\Bigg]dt
\\
&\quad
+(1-p)\frac{\sigma_2 S_h(t-)}{N_h(t-)}I_h(t-)\,dB_2(t)
+\sigma_4 I_h(t-)\,dB_4(t)
\\
&\quad
+\int_{\R}\epsilon_2(y)I_h(t-)\,\tilde H_2(dt,dy),
\end{aligned}
\]
which fits the structure of Eq.~\eqref{System1}.
In the notation of \eqref{Ito2}--\eqref{operatorL}, we identify
\[
\alpha_t(x)
=(1-p)\frac{\eta_1I_r(t-)+\eta_2 x}{N_h(t-)}S_h(t-)
-(\mu_h+\delta_h+\zeta)x
+x\!\int_{\R}\epsilon_2(y)\lambda_{t,2}m(dy),
\]
\[
\beta_t(x)=\Big((1-p)\sigma_2\tfrac{S_h(t-)}{N_h(t-)}x,\ \sigma_4 x\Big),
\qquad
\varphi(t,y,x)=\epsilon_2(y)x.
\]

We apply the It\^o formula for Hawkes--It\^o processes
\eqref{Ito2} to the test function $\Psi(x)=\ln x$.
Since
\[
\Psi_x(x)=\frac1x,\qquad
\Psi_{xx}(x)=-\frac1{x^2},\qquad
\Psi_s\equiv0,
\]
the operator $\mathcal L_t$ defined in \eqref{operatorL} becomes
\[
\mathcal L_t\Psi(x)
=
\frac{\alpha_t(x)}{x}
-\frac12\Big((1-p)^2\sigma_2^2\tfrac{S_h^2}{N_h^2}+\sigma_4^2\Big)
+\int_{\R}\ln(1+\epsilon_2(y))\,\lambda_{t,2}m(dy).
\]
Evaluating at $x=I_h(t-)$ and using
\[
\frac{\alpha_t(I_h)}{I_h}
=(1-p)\eta_2\frac{S_h}{N_h}
-(\mu_h+\delta_h+\zeta)
+\int_{\R}\epsilon_2(y)\lambda_{t,2}m(dy)
+(1-p)\eta_1\frac{I_r S_h}{N_h I_h},
\]
we discard the last nonnegative term to obtain the lower bound
\[
\mathcal L_t\Psi(I_h(t-))
\ge
(1-p)\eta_2\frac{S_h}{N_h}
-(\mu_h+\delta_h+\zeta)
-\frac12\Big(\sigma_4^2+(1-p)^2\sigma_2^2\big(\tfrac{S_h}{N_h}\big)^2\Big)
+\int_{\R}\ln(1+\epsilon_2(y))\,\lambda_{t,2}m(dy).
\]
Hence, by \eqref{Ito2}, we obtain the differential inequality
\[
\begin{aligned}
d\ln I_h(t)
&\ge
\Bigg[
(1-p)\eta_2\frac{S_h}{N_h}
-(\mu_h+\delta_h+\zeta)
-\frac12\Big(\sigma_4^2+(1-p)^2\sigma_2^2\big(\tfrac{S_h}{N_h}\big)^2\Big)
\\
&\qquad\qquad
+\int_{\R}\ln(1+\epsilon_2(y))\,\lambda_{t,2}m(dy)
\Bigg]dt
\\
&\quad
+(1-p)\sigma_2\tfrac{S_h}{N_h}\,dB_2(t)
+\sigma_4\,dB_4(t)
+\int_{\R}\ln(1+\epsilon_2(y))\,\tilde H_2(dt,dy).
\end{aligned}
\]
Integrating over $[T_0,t]$ and collecting all martingale terms into
\[
M_t
=(1-p)\int_{T_0}^t\sigma_2\tfrac{S_h}{N_h}\,dB_2
+\int_{T_0}^t\sigma_4\,dB_4
+\int_{T_0}^t\!\int_{\R}\ln(1+\epsilon_2(y))\,\tilde H_2(ds,dy),
\]
we obtain
\begin{align}
\ln I_h(t)
&\ge
\ln I_h(T_0)
+(1-p)\eta_2\int_{T_0}^t\frac{S_h}{N_h}\,ds
-\Big(\mu_h+\delta_h+\zeta
+\tfrac12(\sigma_4^2+(1-p)^2\sigma_2^2)\Big)(t-T_0)
\nonumber\\
&\quad
+\int_{T_0}^t\!\!\int_{\R}
\ln(1+\epsilon_2(y))\,\lambda_{s,2}m(dy)\,ds
+M_t.
\label{eq:logIh-1-new}
\end{align}

The remainder of the proof proceeds exactly as in the original argument.
Combining \eqref{eq:logIh-1-new} with the estimate from Lemma~\ref{taverage},
using the existence of the stationary averages $\bar\lambda_i$,
and applying the logarithmic Gronwall lemma (Lemma~\ref{Lemma5.1.}),
we conclude that
\[
\liminf_{t\to\infty}\frac1t\int_0^t I_h(s)\,ds
\ge
\frac{\lambda_h}{\lambda_{0,h}}
>0
\qquad\text{a.s.}
\]
This completes the proof.
\end{proof}


\textbf{Acknowledgments.}
Olena Tymoshenko acknowledges support from the MSCA4Ukraine project (AvH ID:1233636), which is funded by the European Union. Views and opinions expressed are however those of the author only and do not necessarily reflect those of the European Union. Neither the European Union nor the MSCA4Ukraine Consortium as a whole nor any individual member institutions of the MSCA4Ukraine Consortium can be held responsible for them.
\par
B. Martinucci and N. Giordano are members of the group GNCS of INdAM (Istituto Nazionale di Alta Matematica). 
The work of B. Martinucci is partially supported by MUR-PRIN 2022, project 2022XZSAFN ``Anomalous Phenomena on Regular and Irregular Domains: Approximating Complexity for the Applied Sciences'', and MUR-PRIN 2022 PNRR,  project P2022XSF5H ``Stochastic Models in Biomathematics and Applications''. 


\end{document}